\documentclass{article}

\usepackage{amsfonts,amsmath,amsthm}
\usepackage{epsfig,graphicx,color}
\usepackage{enumerate}

\usepackage[usenames,dvipsnames,svgnames,table]{xcolor}
\usepackage[colorlinks=true, pdfstartview=FitV, linkcolor=ForestGreen,citecolor=ForestGreen, urlcolor=blue]{hyperref}

\usepackage{a4wide}

\def\R{\mathbb{R}}

\def\eps{\varepsilon}
\def\oG{\mathfrak{G}}

\def\un{\mathbf{1}}

\DeclareMathOperator{\interieur}{int}
\DeclareMathOperator{\epi}{epi}

\newtheorem{theo}{Theorem}[section]
\newtheorem{lem}[theo]{Lemma}
\newtheorem{pro}[theo]{Proposition}

\theoremstyle{definition}
\newtheorem{defi}[theo]{Definition}
\theoremstyle{remark}
\newtheorem{rem}[theo]{Remark}

\numberwithin{equation}{section}

\begin{document}

\title{\bf Quasi-convex Hamilton-Jacobi equations 
posed on junctions: the multi-dimensional case}

\author{C. Imbert\footnote{CNRS \& D\'epartment de Math\'ematiques et  
    Applications, \'Ecole Normale Sup\'erieure (Paris), 45 rue d'Ulm,
    75005 Paris, France}~ and R. Monneau\footnote{70, rue du Javelot,
    75013 Paris, France} }

\maketitle


\begin{abstract}
  A \emph{multi-dimensional junction} is obtained by identifying the
  boundaries of a finite number of copies of an Euclidian
  half-space. The main contribution of this article is the
  construction of a \emph{multidimensional vertex test function}
  $G(x,y)$. First, such a function has to be sufficiently regular to
  be used as a test function in the viscosity solution theory for
  quasi-convex Hamilton-Jacobi equations posed on a multi-dimensional
  junction. Second, its gradients have to satisfy appropriate
  compatibility conditions in order to replace the usual quadratic
  penalization function $|x-y|^2$ in the proof of strong uniqueness
  (comparison principle) by the celebrated doubling variable
  technique. This result extends a construction the authors previously
  achieved in the network setting. In the multi-dimensional setting,
  the construction is less explicit and more delicate.
\end{abstract}


\setcounter{tocdepth}{1}
\tableofcontents

\section{Introduction}

\begin{figure}
\begin{center}
\includegraphics[height=6cm]{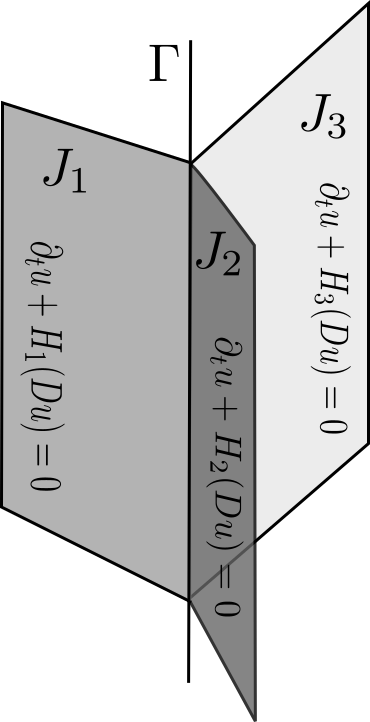}
\caption{A Hamilton-Jacobi equation posed on a multi-dimensional
  junction. Here there are 3 branches (or sheets -- $N=3$) and the
  tangential dimension is $1$ ($d=1$). We did not illustrate the
  junction condition on the junction hyperplane $\Gamma$ (which is a
  line in this example).}
\label{fig:junction}
\end{center}
\end{figure}
A \emph{multi-dimensional junction} is made of a finite number of
copies of an Euclidian half-space glued through their boundaries (see
Figure~\ref{fig:junction}).
\begin{equation}\label{eq::J}
J=\bigcup_{i=1,...,N} J_i \quad \quad
\mbox{with}\quad \left\{\begin{array}{l} J_i = \{ X= (x',x_i): x' \in \R^d, x_i
  \ge 0 \} \simeq \R^{d} \times [0,+\infty) \smallskip\\ 
J_i\cap J_j = \Gamma \simeq \R^d \times \left\{0\right\} \quad \mbox{for}\quad i\not=j
\end{array}\right.
\end{equation}
(with $N \ge 1$ and $d \ge 0$). 
It was previously considered in \cite{fw04} and referred to as an open book. 
It was also considered in \cite{os}. 

The common boundary $\Gamma$ of the half-spaces $J_i$ is referred to
as the {\it junction hyperplane}. For points $X,Y\in J$, the distance
$d(X,Y)$ is defined as follows
\[d^2(X,Y) = \begin{cases} |x'-y'|^2+ (x+y)^2 & \text{ if } X \in J_i, Y \in J_j,
  i\neq j \\ |x'-y'|^2 + |x-y|^2 & \text{ if } X,Y \in J_i.\end{cases}\]

For a sufficiently regular real-valued function $u$ defined on $J$,
$\partial_i u(X)$ denotes the (spatial) derivative of $u$ with respect
to $x_i$ at $X=(x',x_i) \in J_i$ and $D'u (X)$ denotes the (spatial)
gradient of $u$ with respect to $x'$. The ``gradient'' of $u$ is
defined as follows,
\begin{equation}\label{eq:grad}
Du(X):=\begin{cases}
(D'u (X),\partial_i u(X)) & \quad \mbox{if} \quad X\in 
J_i^*:= J_i\setminus \Gamma,\\
(D'u (x',0),\partial_1 u(x',0),...,\partial_N u(x',0)) 
& \quad \mbox{if} \quad X =(x',0)\in \Gamma.
\end{cases}
\end{equation}
With such a notation in hand, we consider a Hamilton-Jacobi equation
posed on the multi-dimensional junction $J$ of the form
\begin{equation}\label{eq:hj-FA}
\left\{\begin{array}{lll}
u_t + H_i(Du)= 0  & t>0, X \in J_i\setminus \Gamma,\\
u_t + F_A(Du)=0   &t>0,  X \in \Gamma
\end{array}\right.
\end{equation}
subject to the initial condition
\begin{equation}\label{eq:ic-0}
u(0,X)=u_0(X) \quad \mbox{for}\quad X\in J.
\end{equation}
The Hamiltonians satisfy the following assumptions.
\begin{equation}\label{assum:H}
\left\{\begin{array}{ll}
\text{\bf (Continuity)} &  H_i \in C(\R^{d+1}) \medskip\\
\text{\bf (Quasi-convexity)} & \forall \lambda, \{ p \in \R^{d+1}: H_i(p) \le \lambda
\} \text{ is convex} \medskip\\
\text{\bf (Coercivity)} & \lim_{|p|\to +\infty} H_i(p)=+\infty.
\end{array}\right.
\end{equation}
The real number $\pi_i^0(p')$ is the minimal $\hat p_i \in \R$ such
that $p_i \mapsto H_i (p',p_i)$ reaches its minimum at
$\hat{p}_i$. The function $H_i^-$ is defined by
\[H_i^-(p',p_i)=\begin{cases}
H_i(p',p_i) &\quad \mbox{if}\quad p_i \le \pi_i^0(p'),\\
H_i(p',\pi_i^0(p')) &\quad \mbox{if}\quad p > \pi_i^0(p')
\end{cases}\] 
(see Figure~\ref{fig:2} in Section~\ref{sec:short}).

The \emph{junction function} $F_A$ appearing in \eqref{eq:hj-FA} is
constructed from the Hamiltonians $H_i$ and a function $A$ defined on
the tangent space of $\Gamma$, referred to as a \emph{flux
  limiter}. After identifying the tangent space of $\Gamma$ with
$\R^d$, flux limiters are functions $A\colon\R^d \to \R$ satisfying
the following assumption.
\begin{equation}\label{assum:Aqc}
\left\{\begin{array}{ll}
\text{\bf (Continuity)} &  A \in C(\R^{d}) \medskip\\
\text{\bf (Quasi-convexity)} & \forall \lambda, \{ p \in \R^{d}: A(p) \le \lambda
\} \text{ is convex}.
\end{array} \right.
\end{equation}
An example of such a flux limiter is given by
\begin{equation}\label{eq:A0}
A_0(p')=\max_{i=1,...,N} A_i (p') \quad \text{ with } \quad A_i (p') = \min_{p_i \in \R} H_i(p',p_i).
\end{equation}
The function $F_A$ is defined as
\begin{equation}\label{eq:FA}
F_A(p',p_1,\dots,p_N)=\max \left(A(p'),  \max_{i=1,...,N} H_i^-(p',p_i)\right)
\end{equation}
(recall the junction condition in \eqref{eq:hj-FA} and the definition
of $Du(x)$ in \eqref{eq:grad} for $x \in \Gamma$).

\paragraph{Main result.} Our main result is the existence of the
\emph{multi-dimensional vertex test function}, a ``sufficiently''
regular function defined on $J^2$ whose gradients satisfy appropriate
compatibility conditions. In the following statement, $C(J)$ and
$C(J^2)$ denote the classes of continuous functions in $J$ and $J^2$
respectively. The class of functions $C^1(J)$ is made of functions of
$C(J)$ such that the restrictions to $J_i$ are $C^1$ up to $\Gamma$ --
see \eqref{eq:C1} below.  
\begin{theo}[The vertex test function]\label{th::G}
  Let $A$ satisfy \eqref{assum:Aqc} with $A\ge A_0$ and let
  $\gamma \in (0,1]$ be a small error parameter. Assume the
  Hamiltonians satisfy \eqref{assum:H}.  Then there exists a function
  $G:J^2\to \R$ enjoying the following properties.
\begin{enumerate}[i)] 
\item \label{i*} \emph{(Regularity)}
\[G\in C(J^2)\quad \text{and}\quad \left\{\begin{array}{l}
G(X,\cdot)\in C^1(J) \quad \text{for all}\quad X\in J,\\
G(\cdot,Y)\in C^1(J) \quad \text{for all}\quad Y\in J.
\end{array}\right.\]
\item \label{ii*} \emph{(Bound from below)} $G\ge 0=G(0,0)$.
\item \label{iii*} \emph{(Compatibility condition on the diagonal)}
For all $X\in J$,
\begin{equation}\label{eq::85}
0\le G(X,X) -G(0,0)  \le \gamma.
\end{equation}
\item \label{iv*} \emph{(Superlinearity)} There exists $g:[0,+\infty)\to \R$
  nondecreasing and such that for $(X,Y)\in J^2$
\begin{equation}\label{eq::20}
g(d(X,Y))\le G(X,Y) \quad \text{and}\quad \lim_{a\to +\infty}
\frac{g(a)}{a} = +\infty.
\end{equation}
\item \label{v*} \emph{(Gradient bounds)} For all $K>0$, there exists  $C_K>0$
  such that for all $(X,Y)\in J^2$,
\begin{equation}\label{eq::19}
d(X,Y)\le K \quad \Longrightarrow \quad |D_X G(X,Y)|+|D_Y G(X,Y)|\le C_K.
\end{equation}
\item \label{vi*}\emph{(Compatibility condition on the gradients)} There exists a
  family of modulus of continuity $\{ \omega_R \}_{R>0}$ such that for
  all $X,Y\in J$ and $K>0$ with $d(X,Y)\le K$,
\begin{equation}\label{eq:shorthandbis}
\left. \begin{array}{ll}
H_j(-D_YG(X,Y))-H_i(D_XG(X,Y)) \le \omega_{C_K}(\gamma C_K) & \text{ if } Y \in J_j^*, X \in J_i^* \\
H_j(-D_YG(X,Y))-F_A(D_XG(X,Y)) \le \omega_{C_K}(\gamma C_K) & \text{ if } Y \in J_j^*, X \in \Gamma\\
F_A(-D_YG(X,Y))-H_i(D_XG(X,Y)) \le \omega_{C_K}(\gamma C_K) &\text{ if } Y \in \Gamma, X \in J_i^*\\
F_A(-D_YG(X,Y))-F_A(D_XG(X,Y)) \le \omega_{C_K}(\gamma C_K) &\text{ if } Y \in \Gamma, X \in \Gamma
\end{array} \right\}
\end{equation}
with $C_K$ given in \eqref{eq::19}.
\end{enumerate}
\end{theo}
\begin{rem}
  We recall that for
  $X\in \Gamma$ (resp. $Y \in \Gamma$), the gradient $D_X G(X,Y)$
  (resp. $D_YG(X,Y)$) is defined in \eqref{eq:grad}.
\end{rem}
Theorem~\ref{th::G} implies strong uniqueness for
\eqref{eq:hj-FA}-\eqref{eq:ic-0}. As a matter of fact, it even implies
strong uniqueness for a large class of Hamilton-Jacobi equations posed
on a generalized junction, see Remark~\ref{rem:gen} for further
details. In order to state the strong uniqueness result for
\eqref{eq:hj-FA}-\eqref{eq:ic-0}, we first make precise in which weak
sense the solutions satisfy the equation and the junction
condition. The appropriate notion is the one of \emph{flux-limited
  solutions} \cite{im}: these solutions are viscosity solutions \`a la
Crandall-Evans-Lions satisfying the junction condition in the strong
viscosity sense. More precisely, they satisfy the equation in the
classical viscosity sense away from the junction hyperplane and they
satisfy the junction condition in the viscosity sense with test
functions that are continuous in $J$ and $C^1$ on each $J_i$ up to
$\Gamma$ -- see Subection~\ref{s.v} for a precise definition.
\begin{theo}[Comparison principle on a multi-dimensional junction]\label{th:comparison}
  Assume that the Hamiltonians $H_i$ satisfy \eqref{assum:H}, the flux
  limiter $A$ satisfies \eqref{assum:Aqc} with $A\ge A_0$ where $A_0$
  is defined in \eqref{eq:A0}, and that the initial datum $u_0$ is
  uniformly continuous.  Then for all flux-limited sub-solution $u$
  and flux-limited super-solution $v$ of
  \eqref{eq:hj-FA}-\eqref{eq:ic-0} satisfying for some $T>0$ and
  $C_T>0$ and $X_0 \in J$,
\begin{equation}\label{eq::27}
\begin{cases} u(t,X)\le C_T (1+ d(X_0,X)),\\ v(t,X)\ge -C_T(1+d(X_0,X)),\end{cases}
\quad
\text{for all}\quad (t,X) \in [0,T)\times J,
\end{equation}
and such that $u(0,X) \le u_0 (X) \le v(0,X)$ for all $X \in J$, we
have
\(u\le v  \text{ in } [0,T)\times J.\)
\end{theo}
\begin{rem} \label{rem:gen}
  A comparison principle holds true for Hamilton-Jacobi equations
  associated with more general junction conditions, see
  \eqref{eq:hj-f} and Assumptions~\eqref{assum:F} and
  \eqref{assum:F-convexity} in Appendix. It is a consequence of the
  fact that imposing general junction conditions reduce to imposing
  $F_A$ ones, see Theorem~\ref{th:class} in Appendix. These results
extend the ones obtained in the one-dimensional setting \cite{im}. 
\end{rem}
\begin{rem}
  Extensions to Hamiltonians depending on $(t,x)$ is not difficult and
  is explained in \cite{im} in the network setting. Such an extension
  is obtained by classically localizing the study around a point
  $(\bar t, \bar x) \in (0,T) \times \Gamma$ at the beginning of the
  proof of the comparison principle. In the remainder of the proof, one
  uses the vertex test function associated with the Hamiltonians whose
  dependence in $(t,x)$ is frozen at $(\bar t, \bar x)$, see \cite{im}
  for details. 
\end{rem}
\begin{rem}
  This comparison principle holds true for semi-solutions growing at
  most linearly, see \eqref{eq::27}. Such a condition is classical for
such equations. 
\end{rem}

\paragraph{Difficulties related to strong uniqueness for \eqref{eq:hj-FA}.}
Getting a strong uniqueness result is known to be difficult for
Hamilton-Jacobi equations such as \eqref{eq:hj-FA}. Indeed, even the
special case $N=2$ is difficult since it corresponds to the study of a
Hamilton-Jacobi equation posed in an Euclidian space whose Hamiltonian
is discontinuous with respect to the space variable along a
hyperplane. More precisely, two different continuous Hamiltonians are
chosen on either side of the hyperplane but they do not coincide on
it. This discontinuity is identified as a major difficulty when
proving a strong uniqueness result such as a comparison principle
(Theorem~\ref{th:comparison}). It is classically proved by the
doubling variable technique: the supremum of $u-v$ in $(0,T) \times J$
is approximated by the supremum of $u(t,x)-v(t,y)-P_\eps(x,y)$ in
$(0,T) \times J \times J$ where $P_\eps (x,y)$ is a penalization
function; the behaviour at infinity of the function $P_\eps(x,y)$ and the
smallness of the parameter $\eps$ force $x$ to be close to
$y$. Classically, $P_\eps (x,y)$ is chosen as the quadratic function
$\eps^{-1} |x-y|^2$; but with such a choice, the proof fails because
of the discontinuity of the Hamiltonian through the
hyperplane. Indeed, two viscosity inequalities are written at points
$(\bar t, \bar x)$ and $(\bar t,\bar y)$ if the approximate supremum
is reached at $(\bar t, \bar x,\bar y)$; if $\bar x$ and $\bar y$ are
not in the same $J_i$, then the Hamiltonians appearing in the two
viscosity inequalities are different.  Some authors impose
compatibility conditions on Hamiltonians but we do not want to do
so. Instead, a natural idea \cite{acct,im} is to design the
penalization function $P_\eps(x,y)$ in such a way that it compensates
the lack of compatibility conditions between Hamiltonians.  Here, it
is chosen in the form $\eps G(x/\eps,y/\eps)$ for some function $G$
referred to as a vertex test function.  The compatibility conditions
on the gradients \ref{vi*}) of $G$ in Theorem~\ref{th::G} address the
lack of compatibility of Hamiltonians.

Apart from the compatibility conditions on the gradients, see
\ref{vi*}), other properties of the vertex test function $G$ are
needed.  The regularity of $G$, see \ref{i*}), allows one to use it
as a test function in $X$ and $Y$. The bound from below, see
\ref{ii*}), the compatibility on the diagonal, see \ref{iii*}), and
the superlinearity, see \ref{iv*}), ensure that $G$ can be used as a
penalization function.  The gradient bounds, see \ref{v*}), are
necessary to handle the unboundedness of the domain.

\paragraph{Difficulties associated with the multidimensional setting.}
The construction of the vertex test function is constructed in two
steps: first an approximate vertex test function is defined, which
satisfies the desired properties except on the set $\{ x=y\}$ of
$J \times J$; second this approximate vertex test function is
regularized on the set $\{ x=y\}$. In the multidimensional setting, each
step is significantly more difficult than in the one-dimensional
setting. When constructing the approximate vertex test function, an
optimization problem with equality constraints has to be solved and
the optimizer is defined implicitly through first order optimality
conditions, while in the one-dimensional setting, this optimization
problem is trivial and the optimizer explicit. As far as the second
step is concerned, it is much more involved to check that the
regularization procedure does not affect the other properties.

\paragraph{Comparison with known results.}
In the special case $N=2$, our results are related to \cite{bbc,bbc2}
where an optimal control problem in a two-domain setting is
studied. In these works, the state of the system evolves according to
two different dynamics on each side of an hypersurface. Moreover, the
two dynamics at the interface corresponding to the maximal and minimal
Ishii's discontinuous solutions of the associated Hamilton-Jacobi
equation are identified. One of the two value functions is
characterized in terms of partial differential equations. We showed in
\cite{im} that, in the one-dimensional setting, both value functions
can be conveniently characterized by using the notion of flux-limited
solutions introduced in \cite{im}. The result of the present paper
indicates that such a connexion holds in the general two-domain
setting, even if this is out of the scope of the present paper.
Moreover, we can deal with quasi-convex Hamiltonians instead of convex
ones.

Achdou, Oudet and Tchou \cite{aot0} use ideas from \cite{bbc,bbc2} to
get a simple proof of the comparison principle on a (one-dimensional)
junction for stationary equations. Then Oudet \cite{os} extended the
results to the multi-dimensional setting, getting a comparison
principle for stationary problems. The reader can observe that this
strong uniqueness result is very similar to the comparison principle
obtained in the present paper; the two works were independent and
achieved approximately at the same time. A two-domain Hamilton-Jacobi
equation of the form \eqref{eq:hj-FA} also appears naturally in the
singular perturbation problem studied in \cite{aot}. 

We would like to mention that the results of \cite{bbc,bbc2} were
recently extended to the general case of stratified spaces in the very
nice paper \cite{bc}. Such results also extend the ones from
\cite{bh}. Some results for discontinuous solutions of Hamilton-Jacobi
equations in stratified spaces can be found in \cite{hz}. In
\cite{csm}, the authors study eikonal equations in ramified spaces.
The reader is also referred to \cite{rz,rsz} for optimal control
problems in multi-domains.  In particular, the authors impose some
transmission conditions.  Up to a certain extent, some of our results
are related to the ones in \cite{gh}, in particular, in the case of
source terms located on hyperplanes. We finally refer the reader to
the numerous references given in \cite{im} and the comments there.

We mentioned that our main motivation for constructing such a vertex
test function is the proof of a comparison principle for
Hamilton-Jacobi equations. Two years after the first version of this
paper was posted, a simpler and alternative proof of this strong
uniqueness result was given in \cite{bbcim}; it is obtained as a
combination of the ideas from \cite{bbc,bbc2,im} and the present paper.
We also recall that the results of the present paper (see
Subsection~\ref{sec:ishii}) are used in \cite{in}.

\begin{rem}
  In a first version of this paper, the material was presented in a
  different way. In order to emphasize the main contribution of the
  present article, we decided to focus the presentation on the
  construction of the vertex test function and the proof of the
  comparison principle for flux-limited solutions and to move into an
  Appendix results related to relaxed solutions. The reason for doing
  so is that the proofs of the results in the Appendix are (more or
  less) a straightforward adaptation from the one-dimensional case.
  The reader is also referred to \cite{in} where the results in the
  Appendix are generalized to the case of degenerate parabolic
  equations. Moreover, the multi-dimensional results of
  Subsection~\ref{sec:ishii} are used in \cite{in}. 
\end{rem}

\paragraph{Organization of the article.} The paper is organized as
follows. We start with the short Section~\ref{sec:short} where
important functions related to Hamiltonians are defined.
Section~\ref{s.G} is devoted to the construction of an approximate
vertex test function in the case where the Hamiltonians are smooth and
convex. Then the main theorem, Theorem~\ref{th::G}, is proved in
Section~\ref{sec:gen}. Section~\ref{sec:flux-limited} contains the
definition of flux-limited solutions and the proof of the comparison
principle (Theorem~\ref{th:comparison}).  Appendix~\ref{sec:appendix}
begins with the definition of relaxed solutions for Hamilton-Jacobi
equations posed on multidimensional junctions
(Subsection~\ref{s.relaxed}).  The classification of general junction
conditions is explained in Subsection~\ref{s23}.
Subsection~\ref{sec:ishii} is devoted to the special case $N=2$ where
maximal and minimal Ishii solutions are related to flux-limited
solutions.

\paragraph{Notation.} The junction hyperplane $\Gamma$ is the common
boundary of $J_i$: we have $\Gamma = \partial J_i$. We identify
$\Gamma$ with $\R^d$ and we do not write the injection of $\R^d$ into
$J_i$: $x' \mapsto (x',0)$.  For this reason, we write
indisctinctively $x=(x',0) \in \Gamma$ and $x' \in \Gamma$.

We set
\begin{equation}\label{eq:C1}
C^1(J)=\left\{\phi\in C(J),\quad \mbox{$\phi$ restricted to $J_i$ is $C^1$ for $i=1,\dots,N$}\right\}.
\end{equation}

 For a function $f : D \to \R$, $\epi f$ denotes its
epigraph $\{ (X,r) \in D \times \R \colon r \ge f(X) \}$. 

\section{Important functions related to  Hamiltonians}
\label{sec:short}

This short section is devoted to the introduction of important
functions that are associated with Hamiltonians $H_i$: the ``natural''
flux limiter $A_0$, the monotone parts $H_i^\pm$, the inverse
functions $\pi_i^\pm$.

The functions $A_0,A_1,\dots, A_N$ are defined in \eqref{eq:A0}, 
\[
A_0(p')=\max_{i=1,...,N} A_i (p') \quad \text{ with } \quad A_i (p') = \min_{p_i \in \R} H_i(p',p_i).
\]
We will prove that the functions $A_i$, $i=0,\dots, N$ are
quasi-convex, continuous and coercive in $p'$ (see
Lemma~\ref{lem:ai-convex} in Appendix).

The Hamiltonian $H_i(p',p_i)$ is defined for 
$p=(p',p_i) \in \R^{d+1}$.  The minimal minimizer of
$p_i \mapsto H_i(p',p_i)$ is denoted by $\pi_i^0(p')$. The
functions $H_i^-$ and $H_i^+$ are defined as follows
\begin{eqnarray*}
 H_i^- (p',p_i) = \begin{cases} H_i (p',p_i) & \text{ if } p_i \le \pi_i^0(p') \\
H_i (p',\pi_i^0 (p')) & \text{ if } p_i \ge \pi_i^0(p')
\end{cases} \\
 H_i^+ (p',p_i) = \begin{cases} H_i (p',p_i) & \text{ if } p_i \ge \pi_i^0(p') \\
H_i (p',\pi_i^0 (p')) & \text{ if } p_i \le \pi_i^0(p').
\end{cases}
\end{eqnarray*}
\begin{figure}
\begin{center}
\includegraphics[height=4cm]{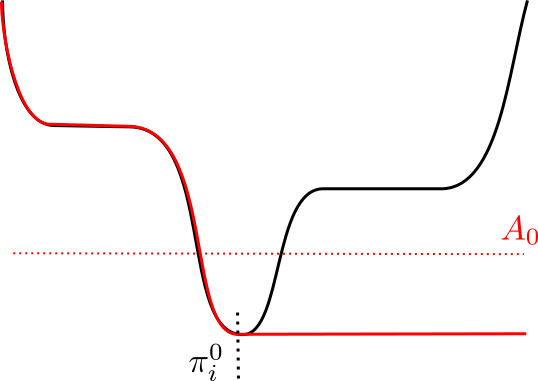} \hspace{4mm}
\includegraphics[height=4cm]{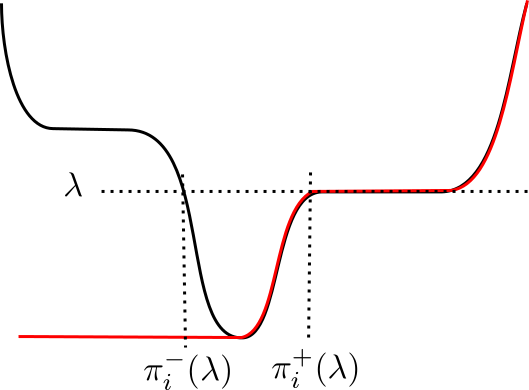}
\caption{Monotone parts $H_i^\pm$ of a Hamiltonian $H_i$ ($H_i^-$ on
  the left, $H_i^+$ on the right). The Hamiltonian is in black,
  monotone parts in red. The tangent variable $p'$ is not shown. In
  this example, the minimum $A_i$ of $H_i$ is lower than $A_0$. The
  ``inverse'' functions $\pi_i^\pm$ of $H_i$ are also shown. }
\label{fig:2} 
\end{center}
\end{figure}
For $\lambda \ge A_i (p')=\min_{p_i \in \R} H_i (p',p_i)$, the functions $\pi_i^\pm$ 
are defined by
\begin{equation} \label{eq::21}
\begin{cases}
\pi_i^+ (p',\lambda) &= \inf \{  p_i : H_i (p',p_i) = H_i^+ (p',p_i) =
\lambda \} \\
  \pi_i^- (p',\lambda) &= \sup \{  p_i : H_i (p',p_i) = H_i^- (p',p_i) =
\lambda \}. 
\end{cases}
\end{equation}

We introduce the  shorthand notation
\begin{equation}\label{eq:shorthand}
H(X,p',p)=\left\{\begin{array}{lll}
H_i(p',p) &\quad \text{for} \quad p=p_i &\quad \text{if} \quad X\in
J_i \setminus \Gamma,\\
F_A(p',p) &\quad \text{for} \quad p=(p_1,...,p_N) &\quad \text{if}
\quad X \in \Gamma.
\end{array}\right.
\end{equation}
In particular, keeping in mind the definition of $Du$ (see
\eqref{eq:grad}), Problem~\eqref{eq:hj-FA} on the junction can be
rewritten as follows
\[u_t + H(X,Du)=0 \quad \text{for all}\quad (t,X)\in (0,+\infty)\times J.\]

\section{Approximate construction in the smooth convex case}
\label{s.G}

This section is devoted to the construction of an approximate vertex
test function $G^0$ in the case where the Hamiltonians and the flux
limiter are smooth and convex. More precisely, we construct a function
$G^0$ that satisfies the desired properties of the vertex test
function except on the subset $\{x=y\}$ of $J \times J$.

We assume throughout this section that the Hamiltonians $H_i$ satisfy
the following assumptions for $i=1,...,N$,
\begin{equation}\label{eq:hi-convex}
\left\{\begin{array}{l}
H_i \in C^2(\R^{d+1}) \quad \text{with}\quad D^2 H_i>0 
\quad \text{in}\quad \R^{d+1},\\
\lim_{|P|\to +\infty} \frac{H_i(P)}{|P|} =+\infty
\end{array}\right.
\end{equation}
and the flux limiter 
\begin{equation}
  \label{eq:a-convex}
 A_0\le  A \in C^2 (\R^d) \quad \text{ and } D^2 A > 0\quad \text{in}\quad \R^{d+1}.
\end{equation}

Recall that $\pi_i^\pm$ are defined in \eqref{eq::21}. 
\begin{lem}[Properties of $\pi_i^\pm$]\label{lem:pi-pm}
  Assume \eqref{eq:hi-convex}. Then $\pi^\pm_i(p',\cdot) \in
  C^2(A_i(p'), +\infty)$ and $\pi^\pm_i \in C(\epi A_i)$. Moreover,
  $\pi^\pm_i$ is concave w.r.t. $(p',\lambda)$ in $\epi A_i$ and $\pm
  \pi^\pm_i$ is non-decreasing w.r.t. $\lambda$.
\end{lem}
\begin{proof}
The regularity of $\pi^\pm$ can be derived   thanks to the inverse function
  theorem. As far as the concavity of $\pi_i^+$ is concerned, we can
  drop the subscript $i$ and we do so for clarity. let
  $(p',\lambda), (q',\mu)  \in
  \epi A$ and $t \in (0,1)$. Then
\begin{align*}
 t \lambda + (1-t) \mu &= t H(p',\pi^+(p',\lambda)) + (1-t) H (q',
\pi^+ (q',\mu)) \\
& \ge H(tp' + (1-t)q', t\pi^+(p',\lambda) + (1-t)\pi^+
(q',\mu)).
\end{align*}
Hence 
\[\pi^+(tp' + (1-t)q',t\lambda + (1-t)\mu) \ge t \pi^+(p',\lambda) + (1-t)\pi^+
(q',\mu)\]
which is the desired result. The monotonicity of $\pi^+$ is easy to
derive from the monotonicity of $H$. The proof of the lemma is now complete. 
\end{proof}

We next define the function $G^0$ for $X \in J_i, Y \in J_j$,
$i,j=1,...,N$, as follows,
\begin{equation}\label{eq:def-g0}
G^0(X,Y)=\sup_{(P,\lambda)\in \mathcal{G}^{ij}_A} (p'\cdot
(x'-y') + p_i x -p_j y -\lambda) 
\end{equation}
where
\begin{equation}\label{eq:germ}
 \mathcal{G}^{ij}_A = 
\begin{cases}
\{ (P,\lambda) \in \R^{d+3} \times \R: P= (p',p_i,p_j), 
 \lambda=H_i(p',p_i)=H_j(p',p_j) \ge
A(p') \} & \text{ if } i \neq j \\
\{ (P,\lambda) \in \R^{d+2} \times \R: P=(p',p_i), 
 \lambda= H_i(p',p_i) \ge A(p') \} & \text{ if } i=j
\end{cases}
\end{equation}
with $A \ge A_0$. 

The main result of this section is the following proposition. 
\begin{pro}[An approximate test function in the smooth convex case]\label{pro:vertex-smooth}
  Let $A\ge A_0$ with $A_0$ given by \eqref{eq:A0} and assume that the
  Hamiltonians satisfy \eqref{eq:hi-convex} and the flux limiter $A$
  satisfies \eqref{eq:a-convex}.  Then $G^0$ satisfies
\begin{enumerate}[i)]
\item \emph{(Regularity)}
\[G^0\in C(J^2)\quad \mbox{and}\quad \left\{\begin{array}{l}
G^0\in C^1(\left\{(X,Y)\in J\times J,\quad x\not=y\right\}),\\
G^0(0,\cdot)\in C^1(J) \quad \mbox{and}\quad G^0(\cdot,0)\in C^1(J);
\end{array}\right.\]
\item \emph{(Bound from below)} $G^0 \ge G^0 (0,0);$
\item \emph{(Compatibility conditions)} \eqref{eq::85} holds with
  $\gamma=0$; and \eqref{eq:shorthandbis} holds with $\gamma =0$ for
  $X=(x',x)$, $Y=(y',y)$ with $x\not=y$ or $x=y=0$;
\item \emph{(Superlinearity)} \eqref{eq::20} holds for some $g=g^0$;
\item \emph{(Gradient bounds)} \eqref{eq::19} holds only for $(X,Y) \in J^2$
  such that $x \neq y$  or $(x,y)=(0,0)$;
\end{enumerate}
\end{pro}
The proof of this proposition is postponed until
Subsection~\ref{subsec:pro80}. 

\subsection{The vertex test function in $J_i \times J_j$ with $i \neq j$}

In order to prove Proposition~\ref{pro:vertex-smooth}, we first need to study the restriction $G^0_{ij}$ of 
$G^0$ to the set $J_i \times J_j$. Then, one can
write 
\[G^0_{ij}(X,Y) = \oG_{ij} (x'-y',x_i,-y_j) \]
with 
\begin{equation}\label{eq:GGG}
\oG_{ij}(Z)=\sup_{(P,\lambda)\in \mathcal{G}^{ij}_A} (P \cdot Z - \lambda)
\end{equation}
where $\mathcal{G}^{ij}_A$ is defined in \eqref{eq:germ}.
Remark that for $X \in J_i$ and $Y \in J_j$, we have $Z= X-Y \in \mathcal Q$ where
\[ \mathcal Q = \R^d \times [0,+\infty[ \times ]-\infty;0].\]
We also consider the simplex 
\[ \mathcal{T} = \{
(\alpha_i,\alpha_j,\alpha_0) \in [0,1]^3: \alpha_i + \alpha_j +
\alpha_0 =1 \}.\]
\begin{lem}[Necessary conditions for the maximiser : $ij$-version]\label{lem:}
Given $Z \in \mathcal Q$, the supremum defining $\oG_{ij}(Z)$ is reached for some
$(P,\lambda) \in \mathcal{G}^{ij}_A$ and there exists
$(\alpha_i,\alpha_j,\alpha_0) \in \mathcal{T}$ such that 
\[ Z = D (\alpha \cdot H) (P) \]
with $H = (H_i,H_j,A)$. 
\end{lem}
\begin{proof}
$\oG_{ij}(Z)$ is defined by maximizing a linear function under an equality
constraint and an inequality constraint. Constraints are qualified if 
\[ D (H_i - H_j) \text{ is not colinear with } D(H_i -A). \]

When constraints are qualified, Karush-Kuhn-Tucker theorem asserts (computing $D_P(P\cdot Z -\lambda)$)
that there exists $\alpha_j \in \R$ and $\alpha_0 \ge 0$ such that 
\[ Z = \nabla_P H_i + \alpha_j (\nabla_P H_j - \nabla_P H_i) + \alpha_0
\nabla_P (A-H_i) \]
with 
\[ \alpha_0 = 0 \quad \text{ if } A(p') < H_i (p',p_i).\] 
If one sets $\alpha_i = 1 - \alpha_0 - \alpha_j$,
 we have equivalently,
\begin{equation*}
  \begin{cases}
    z_i = \alpha_i \partial_i H_i (p',p_i) \ge 0 \\
    z_j = \alpha_j \partial_j H_j (p',p_i) \le 0 \\
    z' = \alpha_i \nabla_{p'} H_i + \alpha_j \nabla_{p'} H_j + \alpha_0
    \nabla_{p'} A 
  \end{cases}
\end{equation*}
The constraints are qualified
in particular  if 
\begin{equation}\label{eq::r11}
 \partial_i H_i (p',p_i) > 0 \text{ and } \partial_j H_j (p',p_j) <0.
\end{equation}
In this case we deduce that  $(\alpha_i,\alpha_j,\alpha_0) \in \mathcal{T}$. 
Hence, the result is proved in case \eqref{eq::r11}.

Now assume that $\partial_i H_i (p',p_i) \le 0$. We remark that in all
cases, $\partial_i H_i(p',p_i) \ge 0$ since $z_i \ge 0$. Hence,
$\partial_i H_i (p',p_i)=0$ or, in other words, $H_i(p',p_i)=A_i
(p')$. But the constraint $H_i(p',p_i) \ge A(p')$, the assumption $A
(p') \ge A_0 (p')$ and the simple fact that $A_i(p') \le A_0 (p')$
imply in particular that $A(p') = A_0 (p')$. We arrive at the same
conclusion if $\partial_j H_j (p',p_j) \ge 0$. In other words,
\begin{equation}\label{eq:constraint-qualif}
 \text{Condition \eqref{eq::r11} holds true as soon as } \quad \forall p', \; A(p') > A_0
 (p').
\end{equation}

In particular, the result of the lemma holds true under this latter
condition: $A(p')>A_0(p')$ for all $p' \in \R^d$. If now there is
some $p'$ such that $A (p')= A_0(p')$, we remark that 
\[ \oG_{ij} (Z) = \lim_{\eps \to 0} \oG_{ij}^\eps (Z) \] where
$\oG_{ij}^\eps (Z)$ is associated with $A^\eps (p') = \eps + A (p')$.
From the previous case, we know that there exists $P_\eps$ and
$\lambda_\eps$ such that
\[ \oG_{ij}^\eps (Z) = P_\eps \cdot Z - \lambda_\eps\]
and $\alpha^\eps = (\alpha_i^\eps,\alpha_j^\eps,\alpha_0^\eps) \in
\mathcal{T}$ such that
\[ Z = D (\alpha \cdot H ) (P_\eps).\]
We can extract a subsequence such that $\alpha^\eps \to
\alpha$. Moreover, $P_\eps \cdot Z - \lambda_\eps$ is bounded from
above and 
\[ \lambda_\eps = H_i (p'^\eps,p_i^\eps) = H_j(p'^\eps,p_j^\eps).\]
Since $H_i$ and $H_j$ are assumed to be superlinear, we conclude that
we can also extract a converging subsequence from $P_\eps$. This
achieves the proof of the lemma. 
\end{proof}
\begin{lem}[Uniqueness of $(P,\lambda)$ : $ij$-version]\label{lem:P-unique}
  Let $Z =(z',z_i,z_j)\in \mathcal Q$. If there exists
  $\alpha,P,\lambda$ and $\beta,Q,\mu$ such that
  $\alpha,\beta\in \mathcal T$ and
\[\begin{cases} 
\oG_{ij} (Z) = P \cdot Z - \lambda = Q \cdot Z -\mu, \\
Z = D ( \alpha \cdot H) (P) = D (\beta \cdot H) (Q).
\end{cases} \] Then $\lambda=\mu$, $p'=q'$ and
\begin{equation}\label{eq::r12}
p_i=q_i=\pi_i^+(p',\lambda)
\end{equation}
except in the case
\begin{equation}\label{eq::r13}
\alpha_i=\beta_i=0=z_i,
\end{equation}
and
\begin{equation}\label{eq::r14}
p_j=q_j=\pi_j^-(p',\lambda)
\end{equation}
except in the case
\begin{equation}\label{eq::r15}
\alpha_j=\beta_j=0=z_j.
\end{equation}
Moreover under the previous assumptions, and in all cases, we can define
$$\hat P = (p',\pi_i^+(p',\lambda),\pi_j^-(p',\lambda))$$
and then we have
$$\oG_{ij} (Z) = \hat P \cdot Z - \lambda \quad\mbox{and}\quad Z= D(\alpha\cdot H)(\hat P)$$
\end{lem}
\begin{proof}
We consider the function $\Psi : \R^{d+2} \times \mathcal{T} \to \R$
defined as follows
\[ \Psi (P,\alpha)  = D (\alpha \cdot H) (P).\]  
By assumption, we have 
\[ 0 = D(\alpha \cdot H)(P)- D(\beta \cdot H)(Q).\]
If $\bar P$ denotes $Q-P$ and $\bar \alpha$ denotes $\beta - \alpha$,
then 
\begin{align*}
 0 &= \int_0^1 \left(\begin{array}{l}
\bar P\\
\bar \alpha
\end{array}\right)\cdot D \Psi (P + \theta \bar P, \alpha + \theta \bar
\alpha) d \theta \\
   &= \int_0^1 D_{P} \Psi (P + \theta \bar P, \alpha + \theta \bar
\alpha) \bar P d\theta + \int_0^1 D_\alpha \Psi (P + \theta \bar P, \alpha + \theta \bar
\alpha) \bar \alpha d   \theta.
\end{align*}
Taking the scalar product with $\bar P$ yields
\begin{align*}
 0 & = \int_0^1 D^2_{PP}  ((\alpha + \theta \bar
\alpha) \cdot H) (P + \theta \bar P) \bar P \cdot \bar P d\theta +
\int_0^1 D_P H (P + \theta \bar P) \bar \alpha \cdot \bar P d \theta
\\
& =  T_1 + T_2 
\end{align*}
with $T_i \ge 0$, $i=1,2$ and
\begin{align*}
T_1 &= \int_0^1 D^2_{PP}  ((\alpha + \theta \bar
\alpha) \cdot H) (P + \theta \bar P) \bar P \cdot \bar P d\theta  \ge
0 \\
T_2 & = \int_0^1 D_P H (P + \theta \bar P) \bar \alpha \cdot \bar P d
\theta \ge 0.
\end{align*}
Indeed, keeping in mind that 
\[ \left\{\begin{array}{l}H_i(P)=H_j(P) \\H_i
    (Q)=H_j(Q) \end{array}\right. \quad \text{ and } \quad 
\begin{cases}
\alpha_0 (A (P) - H_i (P))=0 \\
\beta_0 (A (Q) - H_i (Q))=0 
\end{cases}\] 
we remark that 
\begin{align*}
 \int_0^1 D_P H (P + \theta \bar P) \bar \alpha \cdot \bar P d
\theta & = \bar \alpha \cdot ( H(Q) - H(P))\\
& = \bar \alpha_i (H_i (Q) - H_i (P)) +  \bar \alpha_j (H_j (Q) - H_j
(P)) + \bar \alpha_0 ( A (Q) - A (P) ) \\
& =   (\beta_0 -\alpha_0) ( A (Q) - H_i (Q) - A
(P) + H_i (P) )  \\
& = \beta_0 (H_i(P)-A(P)) + \alpha_0 (H_i(Q)-A(Q)) \ge 0.
\end{align*}
Hence, we get 
\begin{align*}
0& =\int_0^1 D^2_{PP}  ((\alpha + \theta \bar
\alpha) \cdot H) (P + \theta \bar P) \bar P \cdot \bar P d\theta  \\
0 & = \beta_0 (H_i(P)-A(P)) \\
0 &= \alpha_0 (H_i(Q)-A(Q)).
\end{align*}
We distinguish three cases. We will use several times the fact that
$H_i (p',p_i)=\lambda$ and $\partial_i H_i (p',p_i) \ge 0$ implies
that $p_i = \pi_i^+ (p',\lambda)$. We will also use the corresponding
property for $p_j$: $p_j = \pi_j^- (p',p_j)$.
\begin{itemize}
\item \emph{Case 1.} If there exists $\theta \in (0,1)$ such that
  $\alpha + \theta \bar \alpha \in \interieur \mathcal{T}$, then
  $P=Q$ and 
\[\lambda = P \cdot Z - \oG_{ij} (Z) = \mu.\]
\item \emph{Case 2.} If $\alpha = \beta$ is a vertex of $\mathcal{T}$,
  then either $\alpha =(1,0,0)$ or $\alpha = (0,1,0)$ or $\alpha =
  (0,0,1)$. 
\begin{itemize}
\item In the first subcase, $\alpha_i =1$, we get $p'=q'$ and
  $p_i = q_i$ and $Z= \nabla_P H_i(P)$ and 
\[ 0 =  (p_j-q_j) z_j = (P-Q) \cdot Z = \lambda -\mu.\]
We conclude by remarking that we can choose $p_j = \pi^-_j(p',\lambda) = q_j$ when $\alpha_j=\beta_j=0=z_j$.
The second subcase is similar. 
\item If now $\alpha =(0,0,1)$, then $p'= q'$
and $Z = \nabla_P A(P)$ and 
\[ 0 = (p_i-q_i)z_i + (p_j-q_j)z_j = P \cdot Z = \lambda -\mu\]
and we conclude as in the two previous subcases.
\end{itemize}
\item \emph{Case 3.} Assume finally that there exists $\theta \in
  (0,1)$ such that $\alpha + \theta \bar \alpha \in \partial
  \mathcal{T}$ but is not a vertex. In this third case, this implies
  that two components of $a = \alpha + \theta \bar \alpha =
  (a_i,a_j,a_0)$ are not $0$. 
\begin{itemize}
\item If $a_0=0$ then $p'=q'$ and $p_i=q_i$  and $p_j=q_j$, i.e. $P=Q$. 
\item If $a_i=0$ then $p'=q'$ and $p_j=q_j$ and $z_i=0$ and $\lambda =
  \mu$ and we can choose $p_i = \pi^+ (p',\lambda) = q_i$ when $\alpha_i=\beta_i=0=z_i$. The third subcase $a_j=0$
  is similar to the second one. 
\end{itemize}
\end{itemize}
The proof of the lemma is now complete.
\end{proof}
The two previous lemmas imply  the following one. 
\begin{lem}[Gradients of $G^0_{ij}$] \label{lem:Gij-C1}
The function $G^0_{ij}$ is $C^1$ in $J_i\times J_j$, up to the
boundary, and 
\[ D G^0_{ij} (X,Y) = (p',p_i,-p',-p_j), \quad p_i = \pi_i^+ (p',\lambda),
\quad p_j = \pi_j^-(p',\lambda)\quad \mbox{and}\quad P=(p',p_i,p_j)\] where $(p',\lambda)=(\mathfrak{P}
(X,Y),\mathfrak{L} (X,Y))$ are uniquely determined by the relation for some $\alpha\in \mathcal T$
\[ \left\{\begin{array}{l}
G^0_{ij} (X,Y) = p' \cdot (x'-y') + p_i x_i - p_j y_j - \lambda,\\
Z= D(\alpha\cdot H)(P) \quad \mbox{with}\quad Z=(x'-y',x_i,-y_j)
\end{array}\right.\]
In particular, the maps $\mathfrak{P}$ and $\mathfrak{L}$ are
continuous in $J_i \times J_j$.
\end{lem}
The following lemma is elementary but it will be used below.  In view
of the definition of $G^0$, see \eqref{eq:def-g0}, we have the
following equality for $X,Y \in J_i$,
\begin{equation}\label{eq:GGGi}
 G^0_{ii} (X,Y) = (H_i \vee A)^\star (X-Y) 
\end{equation}
where the star exponent denotes here the Legendre-Fenchel
transform. In view of \eqref{eq:GGG}, we also have
the following result.
\begin{lem}[$G^0_{ij}$ at the boundary]\label{lem:oG-boundary}
  The restriction of $\oG_{ij}$ with $i \neq j$ to $\{ z_i =0 \}$ and $\{ z_j =0\}$
  equals respectively $(H_j\vee A)^\star$ and $(H_i\vee A)^\star$.
\end{lem}

\subsection{The vertex test function in $J_i \times J_i$}

We derive from Lemma~\ref{lem:oG-boundary} the
following one. 
\begin{lem}[Continuity of $G^0$]\label{lem:cont}
The function $G^0$ is continuous in $J \times J$. 
\end{lem}
\begin{proof}
The functions $G_{ij}^0$ are continuous by construction since they are convex. 
In order to check that $G^0$ is continuous, it is enough to check that it is
along $\{ z_i =0 \}$. But this is a consequence of Lemma~\ref{lem:oG-boundary} and 
\eqref{eq:GGGi}. 
\end{proof}

We now state the analogues of Lemmas \ref{lem:P-unique},
\ref{lem:Gij-C1} and \ref{lem:oG-boundary}; they are immediately
derived from Formula~\eqref{eq:GGGi}.

\begin{lem}[Necessary conditions for the maximiser : $ii$-version]\label{lem::r16}
  Let $\mathcal T_i$ be defined as follows
  \[\mathcal T_i = \left\{(\alpha_i,\alpha_0)\in [0,1]^2,\quad
    \alpha_i+\alpha_0=1\right\},\]
  and $\alpha\cdot H= \alpha_i H_i + \alpha_0 A$, and $Z=(z',z_i)$. If
  the supremum defining $\oG_{ii}(Z)$ is reached at some
  $(P,\lambda)\in \mathcal G_A^{ii}$, then there exists
  $\alpha\in \mathcal T_i$ such that
$$Z=D(\alpha\cdot H)(P)$$
\end{lem}

\begin{lem}[Uniqueness of $(P,\lambda)$: $ii$-version]\label{lem::r17}
  Let $Z =(z',z_i)\in \R^{d+1}$. If there exists $\alpha,P,\lambda$
  and $\beta,Q,\mu$ such that $\alpha,\beta\in \mathcal T_i$ and
\[\begin{cases} 
\oG_{ii} (Z) = P \cdot Z - \lambda = Q \cdot Z -\mu, \\
Z = D ( \alpha \cdot H) (P) = D (\beta \cdot H) (Q).
\end{cases} \] Then $\lambda=\mu$, $p'=q'$ and
\begin{equation}\label{eq::r12b}
p_i=q_i=\pi_i^+(p',\lambda) \quad \mbox{if}\quad z_i>0
\end{equation}
and
\begin{equation}\label{eq::r14b}
p_i=q_i=\pi_i^-(p',\lambda)  \quad \mbox{if}\quad z_i<0
\end{equation}
Moreover under the previous assumptions, and in all cases, we can define either
$$\hat P = (p',\pi_i^+(p',\lambda)) \quad \mbox{if}\quad z_i\ge0$$
or
$$\hat P = (p',\pi_i^-(p',\lambda)) \quad \mbox{if}\quad z_i\le0$$
and then we always have
$$\oG_{ij} (Z) = \hat P \cdot Z - \lambda \quad\mbox{and}\quad Z= D(\alpha\cdot H)(\hat P)$$

\end{lem}

We now turn to the regularity of $G^0_{ii}$. 
\begin{lem}[Gradients of $G^0_{ii}$]\label{lem:DGii}
$G^0_{ii}$ is $C^1$ in $J_i \times J_i \setminus \{x_i=y_i>0\}$.
For $(X,Y) \in J_i \times J_i$ such that $x_i \neq y_i$, we have
\[ DG^0_{ii} (X,Y) = (p',p_i,-p',-p_i) \quad \mbox{and}\quad P=(p',p_i) \]
with $p_i = \pi_i^\pm (p',\lambda)$ if $\pm (x_i - y_i) > 0$. 
Here $(p',\lambda)=(\mathfrak{P}(X,Y),\mathfrak{L}(X,Y))$ is uniquely determined by 
\[\begin{cases}
 G^0_{ii} (X,Y)  = p' \cdot (x'-y') + p_i (x_i-y_i) - \lambda\\
 Z = \alpha_i D H_i (P) + (1-\alpha_i) DA (P) \quad \mbox{with}\quad Z=(x'-y',x_i-y_i)
\end{cases}\]
which holds true for some $\alpha_i \in [0,1]$. In particular, 
the maps $\mathfrak{P}$ and $\mathfrak{L}$ are continuous in $J_i
\times J_i$. Moreover the restrictions of $G_{ii}^0$ to $(J_i\times J_i) \cap \left\{\pm(x_i-y_i)\ge 0\right\}$
are $C^1$ and
$$G_{ii}^0(x',0,y',0)=p'\cdot(x'-y')-\lambda$$
with
$$DG_{ii}^0(x',0,y',0)=(p',\pi_i^+(p',\lambda), -p',-\pi_i^-(p',\lambda))$$
\end{lem}

\subsection{Proof of Proposition~\ref{pro:vertex-smooth}}

\label{subsec:pro80}

We now turn to the proof of Proposition~\ref{pro:vertex-smooth}.
\begin{proof}[Proof of Proposition~\ref{pro:vertex-smooth}]
The proof proceeds in several steps.

\paragraph{Step 1: Regularity.}
We already noticed in Lemma~\ref{lem:cont} that $G^0\in C(J^2)$ and
Lemmas~\ref{lem:Gij-C1} and \ref{lem:DGii} imply that $G^0 \in C^1(\mathcal R)$
for each region $\mathcal R$ given by
\begin{equation}\label{eq::71}
\mathcal R=\begin{cases}
J_i\times J_j &\quad \text{if}\quad i\not=j,\\
T_i^\pm =\left\{(X,Y) \in J_i\times J_i,\quad \pm (x_i-y_i) \ge 0 \right\} 
 &\quad \text{if}\quad i=j.
\end{cases}
\end{equation}

\paragraph{Step 2: Computation of the gradients.}
For each $\mathcal R$ given by \eqref{eq::71} and for all $(X,Y)\in \mathcal R \subset
J_i\times J_j$, Lemmas~\ref{lem:Gij-C1} and \ref{lem:DGii} imply that
\[G^0(X,Y)= p' \cdot (x'-y') + p_i x_i -p_j y_j -\lambda\]
and
\[ (D',\partial_i) G^0_{|\mathcal R}(X,Y)=(p',p_i)\quad \text{and}\quad 
-(D',\partial_j) G^0_{|\mathcal R}(X,Y)=(p',p_j)\]
with $\lambda=\mathfrak{L}(X,Y)$ and $p'=\mathfrak{P} (X,Y)$ with 
\begin{equation}\label{eq:pipj}
(p_i,p_j)=\left\{\begin{array}{lll}
(\pi^+_i(p',\lambda),\pi^-_j(p',\lambda)) &\quad 
\text{if}\quad \mathcal R=J_i\times J_j &\quad \text{with}\quad i\not=j,\\
(\pi^\pm_i(p',\lambda),\pi^\pm_i(p',\lambda)) &\quad
\text{if}\quad \mathcal R=T^\pm_i &\quad \text{with}\quad i=j. 
\end{array}\right.
\end{equation}
Notice in particular that $\mathfrak{P}$ and $\mathfrak{L}$ are
continuous in $J \times J$.
We also easily deduce that $G^0(X,Y)\ge G^0(X,X)=G^0(0,0)$.

\paragraph{Step 3: Checking the compatibility condition on the gradients.}
Let us consider $(X,Y)\in J^2$, $X=(x',x)$, $Y=(y',y)$ with $x=y=0$ or $x\not= y$.
We have
\begin{align*}
D_X (G^0(\cdot, Y))(X)\in \{(p',\pi^\pm_i(\lambda))\} \\
-(D_Y G^0(X,\cdot))(Y) \in \{(p',\pi^\pm_j(\lambda))\}
\end{align*}
with $\lambda \ge A(p')$. We claim that
\begin{equation}\label{eq:hx}
H(X,D_X G^0(X,Y))=\lambda \quad \mbox{for}\quad N\ge 1
\end{equation}
and
\begin{equation}\label{eq:hy}
H(Y,-D_Y G^0(X,Y))\le \lambda \quad \mbox{for}\quad N\ge 1
\end{equation}
with equality for $N\ge 2$
(we use here once again the short hand notation \eqref{eq:shorthand}).

Equality \eqref{eq:hx} is clear except if $x=0$. In this case, if $y\neq 0$,
say $Y \in J_j$,
the desired equality is rewritten as 
\[ \max (A(p'), \max_i H_i^- (p',p_i)) = \lambda\] with $p_i =
\pi_i^+(p',\lambda)$ if $i \neq j$ and $p_j = \pi_j^-(p',\lambda)$.
Since $\lambda \ge A(p')$ and $H_j^- (p',p_j)=\lambda$, we get the
result for $N\ge 2$. For $N=1$, we have $x-y<0$ and then $p_i=\pi_i^-(p',\lambda)$ which gives again the result.
If now $(x,y)=(0,0)$, then $p_i = \pi_i^+(p',\lambda)$ for all index $i$ and
$\lambda = A(p')\ge A_0(p')$. Hence, we get \eqref{eq:hx} in this case too.\\

One can derive \eqref{eq:hy} in the same way, even with equality for $N\ge 2$.
For $N=1$, where $y=0$, $X=(x',x_i)\in J_i^*$, i.e.  $x_i-y_i>0$, this gives $p_i=\pi_i^+(p',\lambda)$, and we only get
$$H(Y,-D_Y G^0(X,Y)) = \max(A(p'),\min H_i(p',\cdot))\le \lambda$$
with a strict inequality (for $\lambda>A(p')$). On the other hand, we recover equality for $y\not=0$.

\paragraph{Step 4: Superlinearity.}
In view of the definition of $G^0$, we deduce from \eqref{eq:pipj}
that for all $R>0$ and $\lambda > A (R (x'-y')/|x'-y'|)$,
\[G^0(X,Y)\ge R |x'-y'| + \left\{\begin{array}{ll}
x \pi^+_i(R \widehat{x'-y'}, \lambda)-y \pi^-_j(R \widehat{x'-y'},\lambda) - \lambda & \quad \text{if}\quad i\not=j,\\
(x-y) \pi^\pm_i(R \widehat{x'-y'},\lambda) -\lambda & \quad \text{if}\quad i=j, \pm (x-y)\ge 0
\end{array}\right.\]
where $\hat z = z /|z|$. 
For $R>0$, we define
\[ \pi^0(R,\lambda):=\min \{ \eps \pi^\eps_i(p', \lambda) : {\eps \in \{+,-\}, \ i=1,...,N, |p'| \le R} \} \ge
0. \]
Hence we get
\[G^0(X,Y)\ge R |x'-y'|+ \pi^0(R,\lambda)d(x,y)  -\lambda\]
where 
\[ d(x,y) =
\begin{cases}
  |x_i-y_i| & \text{ if } X,Y \in J_i \\
x_i+y_j & \text{ if } X \in J_i, Y \in J_j, i \neq j.
\end{cases}\]

From the definition \eqref{eq::21}  of $\pi^\pm_i$ and the
assumption~\eqref{eq:hi-convex} on the Hamiltonians, we deduce that
\[\pi^0(R,\lambda)\to +\infty \quad \text{as}\quad \lambda\to +\infty\]
and fix some $\lambda(R)\ge \sup_{|p'|\le R} A(p')$ such that $\pi^0(R,\lambda(R))\ge R$. This gives
$$G^0(X,Y)\ge R d(X,Y) -\lambda(R).$$
Therefore we get \eqref{eq::20} with
\[g^0(a)=\sup_{R\ge 0} (Ra - \lambda(R)).\]

\paragraph{Step 5: Gradient bounds.}
Because each component of the gradients of $G^0$ are equal to one of
the $\left\{(p',\pi^\pm_k(p',\lambda))\right\}_{k=1,...,N}$ with
$\lambda= \mathfrak{L}(X,Y)$ and $p' = \mathfrak{P}(X,Y)$, we deduce
\eqref{eq::19} from the continuity of $\mathfrak{L}$, $\mathfrak{P}$
and $\pi^\pm_k$. We use in particular the fact that $\mathfrak{L}$ and $\mathfrak{P}$ only depend on $x'-y'$
and $x_i-y_i$ if $X,Y\in J_i$; and $x'-y'$ and $(x_i,-y_j)$ if $X\in J_i$, $Y\in J_j$ with $i\not= j$.
\end{proof}

\section{Proof of the main theorem}
\label{sec:gen}

\subsection{Proof of  Theorem~\ref{th::G} in the smooth convex case}

With Proposition~\ref{pro:vertex-smooth} in hand, we can now prove
Theorem~\ref{th::G} in the case of smooth convex Hamiltonians.
\begin{lem}[The case of smooth convex Hamiltonians]\label{lem:case-convex}
  Assume that the Hamiltonians satisfy \eqref{eq:hi-convex} and the
  flux limiter $A$ satisfies \eqref{eq:a-convex}.  Then the conclusion
  of Theorem~\ref{th::G} holds true.
\end{lem}
\begin{proof} 
Recall that \eqref{eq:def-g0} can be written as
\[ G^0_{ii}(X,Y)=\oG_{ii}(Z) \quad \mbox{with}\quad Z=X-Y\]
where we recall that $\oG_{ii}$ is defined in \eqref{eq:GGG}. 
Substracting $G^0(0,0)$ to $G^0$ if necessary, we can assume that $G^0(0,0)=0$.
It is enough (and it is our goal) to regularize $G^0_{ii}$ in a neighborhood
  of $\{x_i=y_i\}\backslash \left\{x_i=y_i=0\right\}$.
Let $\eps_0\in (0,1]$ small to fix later, and consider a smooth nondecreasing function $\zeta:\R
    \to [0,1]$ satisfying $\zeta=0$ on $(-\infty,0]$, $\zeta>0$ on $(0,+\infty)$, and $\zeta=1$ on $[B,+\infty)$, with $B\ge 1$ large.
    We also consider a smooth nonincreasing function $\xi : [0,+\infty) \to (0,+\infty)$ with  $\xi(+\infty)=0$,
    which satisfies in particular for $Z=(z',z_i)$ and a real $\bar z_i$
\[
\left|\oG_{ii}(z',z_i)-\oG_{ii}(z',\bar z_i)\right|\le \frac{|z_i-\bar z_i|}{\xi(|z'|)} \quad \mbox{if}\quad |z_i| , |\bar z_i|\le 2\xi(|z'|).
\]
 We will regularize $G^0_{ii}$ in a neighborhood of $\{x=y\}$ of half thickness $\varepsilon_0\theta$ with
 \[
\theta(z',x_i+y_i):= \xi(|z'|)\zeta(x_i+y_i).
\]
 
 To this end, we consider a smooth cut-off function $\Psi:\R \to
        [0,1]$ such that $\mathrm{supp} \, \Psi \subset[-1,1]$ with $\Psi=1$ on $[-1/2,1/2]$.
        We will also use a one-dimensional non-negative  mollifier
        $$\rho_\eta(z_i)=\frac{1}{\eta}\rho(\frac{z_i}{\eta})$$ 
        with $\mbox{supp } \rho\subset [-1,1]$ to regularize by convolution the function $\oG_{ii}(Z)$ in the direction of $z_i$ only,
        because $\oG_{ii}(Z)$ is already $C^1$ in the other directions $z'$.
        Finally we define with $Z=(z',z_i)$ and $z'=x'-y'$, $z_i=x_i-y_i$, the function
        \begin{multline*} G_{ii} (X,Y) = \left(1-\Psi \left(\frac{z_i}{\eps_0
              \theta(z',x_i+y_i)}\right)\right) \oG_{ii}(z',z_i) \\+
        \Psi \left(\frac{z_i}{\eps_0 \theta(z',x_i+y_i)}\right)
        \int_{a\in\R}\rho_{\eps_0 \theta(z',x_i+y_i)}(a)
        \oG_{ii}(z',z_i-a).\end{multline*}
        This regularization procedure preserves the desired properties
        like estimates \eqref{eq::20} (with a possible different
        function $g$ but independent on any $\varepsilon_0\in (0,1]$)
        and \eqref{eq::19} with a possible different constant $C_K$.
        Moreover, for $\varepsilon_0>0$ small enough, this
        regularization procedure introduces a small error $\gamma$ in
        \eqref{eq::85} and another small error $\gamma$ in
        \eqref{eq:shorthandbis}.  This ends the proof of the lemma.
\end{proof}

\subsection{Proof of  Theorem~\ref{th::G} in the general case}

Let us consider a slightly stronger assumption than \eqref{assum:H},
namely
\begin{equation}\label{eq::3}
\left\{\begin{array}{l}
H_i \in C^2(\R^{d+1}) \quad \text{with}\quad \min H_i = H_i (P_i^0)
\quad \text{ and } \quad D^2 H_i(P_i^0)>0,\\
D^2 H_i > 0 \quad \mbox{on}\quad (DH_i)^\perp, \quad \mbox{and}\quad DH_i(P)\not=0 \quad \mbox{for}\quad P\not=P^0_i\\
\displaystyle \lim_{|P|\to +\infty} H_i(P)=+\infty.
\end{array}\right.
\end{equation}
Notice that the second line basically says that the sub-level sets are
strictly convex. The following technical result will allow us 
to reduce a large class of quasi-convex Hamiltonians to convex ones.
\begin{lem}[From quasi-convex to convex Hamiltonians]\label{lem::1}
Given Hamiltonians $H_i$ satisfying \eqref{eq::3}, there exists a function $\beta:\R\to \R$ such that
the functions $\beta \circ H_i$ satisfy \eqref{eq:hi-convex} for
$i=1,...,N$.  Moreover, we can choose $\beta$ such that
\begin{equation}\label{eq::23}
\beta \quad \text{is convex},\quad \beta\in C^2(\R) 
\quad \text{and}\quad \beta'\ge \delta >0.
\end{equation}
\end{lem}
\begin{proof}
In view of \eqref{eq::3}, it is easy to check that $D^2(\beta \circ
H_i)> 0$ if and only if we have
\begin{equation}\label{eq::22}
0< \left\{(\ln \beta')'(\lambda)\right\} \left(\widehat{DH_i}\otimes \widehat{DH_i}\right) \circ \pi_i^\pm(p',\lambda) + \frac{D^2 H_i}{|DH_i|^2}\circ
  \pi_i^\pm(p',\lambda)\quad \text{for}\quad \lambda > H_i(P_i^0), \quad p' \in \R^d.
\end{equation}
Because $D^2 H_i(P_i^0)>0$, we see that the right hand side is positive for
$\lambda$ close enough to $H_i(P_i^0)$.  Then it is easy to choose a
function $\beta$ satisfying \eqref{eq::22} and \eqref{eq::23} (looking at each level set $\left\{H_i=\lambda\right\}$).
Finally, compositing $\beta$ with another convex increasing function
which is superlinear at $+\infty$ if necessary, we can ensure that
$\beta\circ H_i$ superlinear.
\end{proof}
\begin{lem}[The case of smooth Hamiltonians]\label{lem:case-smooth}
  Theorem~\ref{th::G} holds true if the Hamiltonians satisfy
  \eqref{eq::3}.
\end{lem}
\begin{proof}
We assume that the Hamiltonians $H_i$ satisfy \eqref{eq::3}. Let
$\beta$ be the function given by Lemma~\ref{lem::1}. If $u$ solves
\eqref{eq:hj-FA} on $J_T$, then $u$ is also a 
solution of
\begin{equation}\label{eq::1ter}
\left\{\begin{array}{lll}
\bar \beta (u_t) + \hat{H}_i(Du)= 0  &\text{for}\quad t\in (0,T) &\quad \text{and}\quad X\in J_i^*,\\
\bar \beta (u_t)  + \hat{F}_{\hat{A}}(Du)=0   &\text{for}\quad  t\in (0,T) &\quad \text{and}\quad X \in \Gamma
\end{array}\right.
\end{equation}
with $\hat{F}_{\hat{A}}$ constructed as $F_A$ where $H_i$ and $A$ are replaced with
$\hat{H}_i$ and $\hat{A}$ defined as follows
\[\hat{H}_i = \beta\circ H_i,\quad \hat{A} = \beta(A)\]
and $\bar \beta(\lambda)=-\beta(-\lambda)$.  We can then apply
Theorem~\ref{th::G} in the case of smooth convex Hamiltonians to
construct a vertex test function $\hat{G}$ associated to problem
\eqref{eq::1ter} for every $\hat{\gamma}>0$.  This means that we have
with $\hat{H}(X,P) = \beta(H(X,P))$,
\[\hat{H}(Y, -D_Y G) \le \hat{H}(X,D_X G) +\hat{\gamma}.\]
This implies
\[{H}(Y, -D_YG) \le \beta^{-1}(\beta({H}(X,D_X G)) + \hat{\gamma}) \le
{H}(X,D_X G) + \hat{\gamma} |(\beta^{-1})'|_{L^\infty(\R)}.\]
Because of the lower bound on $\beta'$ given by Lemma~\ref{lem::1}, 
we get $|(\beta^{-1})'|_{L^\infty(\R)}\le 1/\delta$ which yields 
the compatibility condition~\eqref{eq:shorthandbis}
with $\gamma= \hat{\gamma}/\delta$ arbitrarily small.
\end{proof}
We are now in position to prove Theorem~\ref{th::G} in the general
case. 
\begin{proof}[Proof of Theorem~\ref{th::G}]
Let us now assume that the Hamiltonians only satisfy
\eqref{assum:H}. In this case, we  approximate the Hamiltonians
$H_i$ by other Hamiltonians $\tilde{H}_i$ satisfying \eqref{eq::3}
such that
\[|H_i-\tilde{H}_i|\le \gamma.\]
Smoothness ($C^2$) is obtained by a standard mollification. It does
not affect quasi-convexity and coercivity.  The condition
$D^2\tilde{H}_i (P_i^0)$ is easily obtained by adding a small
``localized'' $C^2$ quasi-convex function satisfying this condition since
$D^2 H_i (P_i^0) \ge 0$. In order to ensure that there is no critical
point apart from $P_i^0$ and that level sets are strictly convex
($D^2\tilde{H}_i >0$ in $(D\tilde{H}_i)^\perp$), another small
$C^2$ quasi-convex function is added.

We then apply Theorem~\ref{th::G} to the Hamiltonians $\tilde{H}_i$
and construct an associated vertex test function $\tilde{G}$ also for
the parameter $\gamma$.  We deduce that
\[H(Y, -\tilde{G}_Y) \le H(X,\tilde{G}_X) + 3\gamma\]
with $\gamma>0$ arbitrarily small, which shows again the compatibility
condition on the Hamiltonians \eqref{eq:shorthandbis} for the Hamiltonians
$H_i$'s. The proof is now complete in the general case.
\end{proof}

\section{Flux-limited solutions on a multi-dimensional junction}
\label{sec:flux-limited}

\subsection{Flux-limited solutions}
\label{s.v}

For $T>0$, set $J_T= (0,T)\times J$.  In order to define flux-limited
solutions, we first make precise the relevant class of test functions,
\begin{equation}\label{eq::r1}
C^1(J_T)=\left\{\varphi\in C(J_T),\; \varphi \text{ restricted to $(0,T)\times J_i$ 
is $C^1$ for $i=1,...,N$}\right\}.
\end{equation}
We also recall the definition of
upper and lower semi-continuous envelopes $u^*$ and $u_*$ of a
(locally bounded) function $u$ defined on $[0,T)\times J$:
\[u^*(t,X)=\limsup_{(s,Y)\to (t,X)} u(s,Y)
\qquad \text{and}\qquad u_*(t,X)=\liminf_{(s,Y)\to (t,X)} u(s,Y).\] 
\begin{defi}[Flux-limited solutions]\label{defi:fl}
  Assume the Hamiltonians satisfy \eqref{assum:H} and the flux limiter
  $A:\R^d \to \R$ is continuous. Let $u:[0,T)\times J\to \R$ be locally
  bounded.
\begin{enumerate}[i)]
\item We say that $u$ is a  \emph{$A$-flux-limited sub-solution}
  (resp. \emph{$A$-flux-limited  super-solution}) of \eqref{eq:hj-FA} in $J_T$
  if for all test function $\varphi\in C^1(J_T)$ such that
\[ 
u^*\le \varphi \quad (\text{resp.}\quad u_*\ge \varphi) 
\quad \text{in a neighborhood of $(t_0,X_0)\in J_T$}\]
with equality at $(t_0,X_0)$ for some $t_0>0$, we have
\begin{eqnarray}
\nonumber 
\varphi_t + H_i(D\varphi) \le 0 & \quad (\text{resp.}\quad \ge 0) 
\quad \text{at $(t_0,X_0)$ \qquad if $X_0\in J_i^*=J_i \setminus \Gamma$}\\
\label{eq::10}
\varphi_t + F_A(D\varphi) \le 0 & \quad (\text{resp.}\quad \ge 0) \quad
\text{at $(t_0,X_0)$ \qquad if $X_0 \in \Gamma$}.
\end{eqnarray}
\item  We say that $u$ is a \emph{$A$-flux-limited solution}  of \eqref{eq:hj-FA} if $u$ is
  both a $A$-flux-limited sub-solution and a $A$-flux-limited super-solution of \eqref{eq:hj-FA}.
\end{enumerate}
\end{defi}

\subsection{Proof of Theorem~\ref{th:comparison}}
\label{sec:comparison}

We now prove the comparison principle for \eqref{eq:hj-FA},
Theorem~\ref{th:comparison}. It implies in particular that the
$F$-relaxed solution given by Theorem~\ref{th:existence} is unique.
The proof follows the lines of the corresponding one in the
one-dimensional setting \cite{im}. The following elementary a priori
estimate is needed.
\begin{lem}[A priori control]\label{lem:apriori} 
  For $u$ and $v$ as in the statement of Theorem~\ref{th:comparison},
  there exists $C>0$ such that for all $(t,X), (s,Y) \in (0,T) \times J$, 
\begin{equation}\label{eq:diff-times}
 u(t,X) \le v(s,Y) + C (1+d(X,Y))).
\end{equation}
\end{lem}
\begin{proof}
The proof proceeds in several steps. 
\medskip

\textsc{Barriers.}  Since $u_0$ is uniformly continuous, there exists
$u_0^\eps$ which is Lipschitz continuous and such that
\[ |u_0^\eps - u_0 | \le \eps. \]
We remark that 
\[ U_\eps^\pm (t,X) = u_0^\eps (x) \pm C t \pm \eps \] 
is a super-(resp. sub-)solution of \eqref{eq:hj-f}, \eqref{eq:ic} if 
$C$ is chosen large enough. 
\medskip

\textsc{Control at the same time.} We first prove that for $(t,X) \in (0,T) \times J$, 
\begin{equation}\label{eq:same-time}
 u(t,X) \le v(t,Y) + C_1 (1+d(X,Y)). 
\end{equation}
In order to get such an estimate, we consider 
\[ \phi (X,Y) = (1+ d^2 (X,Y))^{\frac12}.\]
It is $C^1$ in $J^2$ and $1$-Lipschitz continuous. 
We then consider 
\[ M = \sup_{t \in (0,T), X,Y \in J} u(t,X) -v(t,Y) - C_{1,1} t
-C_{1,2} \phi (X,Y) - \frac\eta{T-t} - \alpha d^2 (X_0,X) \]
for some $X_0 \in J$.  Our goal is to prove that $M \le 0$ for
$C_{1,1}$ and $C_{1,2}$ sufficiently large (independently of $\eta$
and $\alpha$ in $(0,1)$, say).  Since $u$ and $v$ are sub-linear, see
\eqref{eq::27}, we have
\[ u(t,X)-v(t,Y) \le C_T (2+ d(X_0,X)+ d(X_0,Y)).\]
In particular, the supremum $M$ is reached as soon as $C_{1,2} > C_T$. 
Since $u_0$ is uniformly continuous, there exists $C_0>0$ such that 
\[ u_0(X) - u_0(Y) \le C_0 \phi (X,Y). \]
In particular, if $C_{1,2}>C_0$, we are sure that the supremum is reached
for some $t>0$. 

We next explain why 
\begin{equation}\label{eq:desired-bis}
\alpha d(X_0,X) \le 2C_T (1+ C_T) = \tilde{C}_T
\end{equation}
for $X$ realizing the
supremum $M$. We have 
\begin{align*}
 C_{1,2} \phi(X,Y) + \alpha d^2 (X_0,X) & \le u(t,X) - v (t,Y) \\
& \le C_T  (2+ d(X_0,X)+ d(X_0,Y)) \\
& \le C_T (2 + 2 d(X_0,X) + \phi (X,Y)).
\end{align*}
In particular, with $C_{1,2} > C_T$, we get
\[ \alpha d^2 (X_0,X) \le 2C_T (1 + d(X_0,X)) \]
which yields \eqref{eq:desired-bis}. 

We now write the two viscosity inequalities. There exists $a,b \in \R$ with $a-b= C_{1,1}+ \eta (T-t)^2$ such that
\begin{align*}
a + H (X, C_{1,2} \phi_X (X,Y) + 2\alpha d(X_0,X)) \le 0 \\
b + H (Y,-C_{1,2} \phi_Y(X,Y)) \ge 0 
\end{align*}
where we abuse notation by writing $2\alpha d(X_0,X)$ instead of $2 \alpha d(X_0,X) n(X)$ with $n(X)=\pm 1$. Substracting these inequalities yields
\[
 C_{1,1}  \le H (Y,-C_{1,2} \phi_Y(X,Y))-H (X, C_{1,2} \phi_X (X,Y) + 2\alpha d(X_0,X)).
\]
We finally remark that the right hand side is bounded by a constant depending on $C_{1,2}$. 
We thus can choose $C_{1,1}$ large enough to reach the desired contradiction. 
\medskip

\textsc{Control at different times.} We now derive \eqref{eq:diff-times} from
the barriers constructed above and \eqref{eq:same-time}. Remark that 
\[ U_\eps^+ (t,Y) - U_\eps^-(s,X) \le L_\eps d(X,Y) + 2CT + 2 \eps \le C_2 (1+ d(X,Y)).\]
Applying \eqref{eq:same-time} to $u$ and $U_\eps^+$ and then to $U_\eps^-$ and $v$, we get 
\begin{align*}
 u(t,X) \le U_\eps^+ (t,Y) +  C_1 (1+d(X,Y)) \\
U_\eps^- (s,X) \le v (s,Y) + C_1 (1+ d(X,Y))
\end{align*}
Combining the three previous inequalities yields the desired result. 
\end{proof}

\begin{proof}[Proof of Theorem~\ref{th:comparison}]
Our goal is to prove that 
\[ M = \sup_{t \in (0,T), X \in J} u(t,X) - v(t,X) \le 0. \]
We argue by contradiction and assume that $M>0$. This implies that for $\eta$ and $\alpha$ small enough, 
we have for all $\eps>0$, $\nu>0$ that $M_{\eps,\alpha} \ge \frac{3M}4 > 0$ where
\[ M_{\eps,\alpha} = \sup_{(t,X), (s,Y) \in (0,T) \times J} u(t,X) -v(s,Y) - \eps G (\eps^{-1}X,\eps^{-1}Y)
- \frac{(t-s)^2}{2\nu} - \frac\eta{T-t} - \alpha d^2 (X_0,X) \]
where $G$ is the vertex test function given by Theorem~\ref{th::G} with $\gamma$ to be chosen. 
\medskip

Since $M_{\eps,\alpha}$ is larger than $3M/4$, we can restrict the supremum to points $(t,X)$, $(s,Y)$ such that 
\begin{equation}\label{eq:XY}  u(t,X) -v(s,Y) - \eps G (\eps^{-1}X,\eps^{-1}Y)
- \frac{(t-s)^2}{2\nu} - \frac\eta{T-t} - \alpha d^2 (X_0,X) \ge M/2.
\end{equation}
In particular, thanks to \eqref{eq::20} and Lemma~\ref{lem:apriori}, these points satisfy
\[\eps g\left(\frac{d(X,Y)}\eps \right) \le C (1 + d(X,Y)).\]
Since $g$ is super-linear, we have
\[ d(X,Y) = \omega (\eps) \]
for some modulus of continuity $\omega$ depending on $g$ and $C$. 
We can also derive from \eqref{eq:XY} and Lemma~\ref{lem:apriori} that 
\begin{equation}\label{eq:alpha}
 \alpha d^2(X_0,X) \le C (1 + d(X,Y)) \le C (1+ \omega (\eps)).
\end{equation}
In particular, the points satisfying \eqref{eq:XY} are such that $X$ and $Y$ are bounded by a constant
depending on $\alpha$; this implies that $M_{\eps,\alpha}$ is reached at points we keep denoting by $(t,X)$ and 
$(s,Y)$. 
\medskip

Assume that there exists a sequence $\nu_n \to 0$ such that the corresponding points $(t_n,X_n)$ and $(s_n,Y_n)$ 
are such that $t_n=0$ or $s_n=0$. If $(X_0,Y_0)$ is an accumulation point of $(X_n,Y_n)$, we have 
\[ 0 < \frac{M}2 \le u_0 (X_0) - u_0(Y_0) \le \omega_0 (d(X_0,Y_0)) \le \omega_0 (\omega (\eps))\]
where $\omega_0$ is the modulus of continuity of $u_0$. This implies a contradiction by choosing $\eps$ small. 
\medskip

We conclude that for $\nu$ small enough, we have $t>0$ and $s>0$ and that we can write two viscosity inequalities. 
\begin{align*}
\frac\eta{T^2} + \frac{t-s}\nu + H (X, G_X (\eps^{-1}X,\eps^{-1}Y) + \alpha d (X_0,X)) \le 0 \\
 \frac{t-s}\nu + H (Y, -G_Y (\eps^{-1}X,\eps^{-1}Y)) \le 0 
\end{align*}
where we abuse notation by writing $\alpha d(X_0,X)$. Substracting
these inequalities and using \eqref{eq:shorthandbis}, we get
\[ \frac\eta{T^2} \le H (X,G_X (\eps^{-1}X,\eps^{-1}Y)) - H (X,G_X (\eps^{-1}X,\eps^{-1}Y)+\alpha d(X_0,X)) +
\omega_{C_{K_\eps}} (\gamma C_{K_\eps}) 
\]
where $K_\eps = \eps^{-1} \omega(\eps)$. Letting $\alpha \to 0$, we get from \eqref{eq:alpha} that 
$\alpha d (X_0,X) \to 0$ and letting $\gamma \to 0$, we get $\omega_{C_{K_\eps}} (\gamma C_{K_\eps}) \to 0$. 
These limits  imply the following contradiction $\frac\eta{T^2} \le 0$. 
\end{proof}

\appendix

\section{Relaxed  solutions, effective junction conditions and Ishii solutions}
\label{sec:appendix}

This appendix contains additional results about another notion of
viscosity solutions on a multi-dimensional junction, relaxed
solutions. As explained in \cite{im}, it is easy to construct relaxed
solutions (Theorem~\ref{th:existence} below) while it is possible to
prove uniqueness of flux-limited ones
(Theorem~\ref{th:comparison}). These notions turn out to coincide:
relaxed solutions associated to a flux function $F$ coincide with
flux-limited ones for a flux limiter only depending on the $H_i$ and
$F$ (Theorem~\ref{th:class}). Minimal and maximal Ishii solutions are
also identified (Proposition~\ref{prop:fl-is}).

The main reason for putting such results in appendix is that they are
expected from the one-dimensional setting and/or their proofs
are very similar to the one-dimensional setting.

\subsection{Relaxed solutions on a multi-dimensional junction}
\label{s.relaxed}

We consider Hamilton-Jacobi equations posed on $J$, associated with
general junction function $F: \R^d \times \R^N \to \R$,
\begin{equation}\label{eq:hj-f}
\left\{\begin{array}{lll}
u_t + H_i(Du)= 0  & t>0, X \in J_i\setminus \Gamma,\\
u_t + F(Du)=0   & t>0,  X \in \Gamma
\end{array}\right.
\end{equation}
subject to the initial condition
\begin{equation}\label{eq:ic}
u(0,X)=u_0(X) \quad \mbox{for}\quad X\in J.
\end{equation}
The second equation in \eqref{eq:hj-f} is referred to as \emph{the
  junction condition}.

As far as general junction conditions are concerned, we assume that the
junction function $F:\R^d \times \R^N \to \R$ satisfies 
\begin{equation}\label{assum:F}
\left\{\begin{array}{ll}
\textbf{(Continuity)} & F \in C(\R^d \times \R^N) \medskip \\
\textbf{(Monotonicity)} & \forall i, \; p_i \mapsto F(p',p_1,\dots,p_N)
 \text{ is non-increasing} 
\end{array}\right.
\end{equation}
and, in some important cases,
\begin{equation}\label{assum:F-convexity}
 \textbf{(Quasi-convexity)} \qquad  \forall \lambda, \; \{ p \in \R^d \times \R^N: F(p) \le \lambda \}
 \text{ is convex.}
\end{equation}
\begin{lem}
If the Hamiltonians satisfy \eqref{assum:H} and $A$ satisfies \eqref{assum:Aqc}, then $F_A$ 
defined in (\ref{eq:FA}) satisfies (\ref{assum:F}) and (\ref{assum:F-convexity}).
\end{lem}
\begin{proof}
  Condition~\eqref{assum:F} is clear since $A$ and $H_i^-$ are
  continuous and have the desired monotonicity property. As far as
  \eqref{assum:F-convexity} is concerned, we have to justify that 
  \begin{equation}\label{eq:hic}
 \{ (p',p_i) : H_i^-(p',p_i) \le \lambda \} \text{ is convex.}
\end{equation}
  Indeed, if this holds true then $F_A$ is the maximum of functions
  with convex sub-level sets and it thus also enjoys such a property.
In order to get \eqref{eq:hic}, we remark that the definition of $H_i^-$ implies 
that 
\[  \{ (p',p_i) : H_i^-(p',p_i) \le \lambda \} = \{ (p',p_i) : H_i(p',p_i) \le \lambda \} + 
\{ 0_{\R^d} \} \times [0,+\infty).\]
Since the sum of two convex sets is convex, we indeed have \eqref{eq:hic}.  
\end{proof}

\begin{defi}[Relaxed  solutions]\label{defi:relaxed}
  Assume the Hamiltonians satisfy \eqref{assum:H} and the flux
  function $F$ satisfies \eqref{assum:F}. Let $u:[0,T)\times J\to \R$
  be locally bounded.
\begin{enumerate}[i)]
\item We say that $u$ is an \emph{$F$-relaxed sub-solution}
  (resp. \emph{$F$-relaxed super-solution}) of \eqref{eq:hj-f} in $J_T$
  if for all test function $\varphi\in C^1(J_T)$ such that
\[ u^*\le \varphi \quad (\text{resp.}\quad u_*\ge \varphi) 
\quad \text{in a neighborhood of $(t_0,X_0)\in J_T$}\]
with equality at $(t_0,X_0)$ for some $t_0>0$, we have
\[\varphi_t + H_i(D\varphi) \le 0  \quad (\text{resp.}\quad \ge 0) 
\quad \text{at } (t_0,X_0)\]
 if $X_0 \in J_i^*$, and 
\[\left.\begin{array}{lll}
\text{either } & 
\varphi_t + F(D\varphi) \le 0  &\quad (\text{resp.}\quad \ge 0)  \\
\text{or } &
\varphi_t + H_i(D \varphi) \le 0 & \quad (\text{resp.}\quad
\ge 0) \quad \text{ for some } i
\end{array} \right| 
\quad \text{at } (t_0,X_0)\]
 if $X_0 \in \Gamma$.
\item We say that $u$ is an \emph{$F$-relaxed solution} of
  \eqref{eq:hj-f} if $u$ is both an $F$-relaxed sub-solution of
  \eqref{eq:hj-f} and an $F$-relaxed super-solution of
  \eqref{eq:hj-f}.
\end{enumerate}
\end{defi}
We observe that any $A$-flux-limited solution of \eqref{eq:hj-FA} is also 
a $F_A$-relaxed solution of \eqref{eq:hj-FA}. The following proposition 
asserts that the converse is also true. 
\begin{pro}[Relaxed and flux-limited solutions coincide for
  flux-limited junction conditions]
\label{pro::1}
Assume the Hamiltonians satisfy \eqref{assum:H} and consider a
continuous flux limiter $A$. If $F=F_A$, then relaxed
(sub-/super-)solutions of \eqref{eq:hj-FA} are flux-limited
(sub-/super-)solutions of \eqref{eq:hj-FA}.
\end{pro}
\begin{proof}
We treat successively the super-solution case and the sub-solution
case. 

Let $u$ be a relaxed super-solution and let us assume by contradiction
that there exists a test function $\varphi$ touching $u_*$ from below at
$P_0=(t_0,X_0)$ for some $t_0\in (0,T)$ and $X_0 \in \Gamma$, such that
\begin{equation}\label{eq::11}
\varphi_t + F_A(D\varphi) <0 \quad \text{at}\quad P_0.
\end{equation}
Consider next the test function $\tilde{\varphi}$ satisfying
$\tilde{\varphi}\le \varphi$ in a neighborhood of $P_0$, with equality
at $P_0$ such that
\[\begin{array}{rl}
\tilde{\varphi}_t(P_0) &= \varphi_t(P_0) \\
D'\tilde{\varphi}(P_0) &= D'\varphi (P_0)
\end{array}
\quad \text{and}\quad 
\partial_i \tilde{\varphi}(P_0) = 
\min(\pi^0_i(D'\varphi(P_0)), \partial_i \varphi(P_0))
\quad \text{for}\quad i=1,...,N.\] Using the fact that
$F_A(D\varphi)=F_A(D\tilde{\varphi})\ge H_i^-(D'\tilde{\varphi},\partial_i
\tilde{\varphi}) = H_i(D'\tilde{\varphi},\partial_i \tilde{\varphi})$ at $P_0$ for all $i$, we
deduce a contradiction with \eqref{eq::11} using the viscosity
inequality satisfied by $\tilde{\varphi}$ for some $i \in \{1,\dots,N\}$.

Let now $u$ be a relaxed sub-solution and let us assume by
contradiction that there exists a test function $\varphi$ touching $u^*$
from above at $P_0=(t_0,X_0)$ for some $t_0\in (0,T)$ and $X_0 \in
\Gamma$, such that
\begin{equation}\label{eq::12}
\varphi_t + F_A(D\varphi) >0 \quad \text{at}\quad P_0.
\end{equation}
Let us define
\[I=\left\{i\in \left\{1,\dots,N\right\},\quad
H_i^-(D'\varphi,\partial_i \varphi) <
F_A(D\varphi) \quad \text{at}\quad P_0\right\}\]
and for $i\in I$, let $q_i\ge \pi_i^0(D'\varphi(P_0))$ be such that
\[H_i(D'\varphi(P_0),q_i)=F_A(D\varphi(P_0))\]
where we have used the fact that
$H_i(D'\varphi(P_0),+\infty)=+\infty$.  Then we can construct a test
function $\tilde{\varphi}$ satisfying $\tilde{\varphi}\ge \varphi$ in
a neighborhood of $P_0$, with equality at $P_0$, such that
\[\begin{array}{rl} \tilde{\varphi}_t(P_0) &= \varphi_t(P_0) \\
D' \tilde{\varphi} (P_0) & = D'\varphi (P_0)
\end{array} \quad \text{and}\quad 
\partial_i \tilde{\varphi}(P_0) = \left\{\begin{array}{ll}
\max(q_i, \partial_i \varphi(P_0)) &\quad \text{if}\quad i\in I,\\
\partial_i \varphi(P_0) &\quad \text{if}\quad i\not\in I.
\end{array}\right.\]
Using the fact that $F_A(D\varphi) = F_A(D\tilde{\varphi})\le
H_i(D'\tilde{\varphi},\partial_i \tilde{\varphi})$ at $P_0$ for all $i$, we deduce a
contradiction with \eqref{eq::12} using the viscosity inequality for
$\tilde{\varphi}$ for some $i \in \{1,\dots,N\}$.
\end{proof}
The notion of relaxed solutions given in the previous subsection is
chosen so that it enjoys good stability results; in particular,
existence follows by Perron's method \cite{I}. More details are given
in a more general setting in \cite{in}. 
\begin{theo}[Existence]\label{th:existence}
Let $T>0$.  Assume that Hamiltonians satisfy \eqref{assum:H}, that
the junction function $F$ satisfies \eqref{assum:F} and that the
initial datum $u_0$ is Lipschitz continuous in $J$. Then there exists a
relaxed  solution $u$ of \eqref{eq:hj-f}-\eqref{eq:ic} in
$[0,T)\times J$ and a constant $C_T>0$ such that
\[|u(t,X)-u_0(X)|\le C_T \quad \text{for all}\quad (t,X)\in [0,T)\times J.\]
Moreover $u$ is unique and continuous.
\end{theo}

\subsubsection{The ``weak continuity'' condition for sub-solutions}

If $F$ not only satisfies \eqref{assum:F}, but is also semi-coercive,
that is to say if
\begin{equation}\label{eq::r20}
F(p',p)\to +\infty \quad \mbox{as}\quad \min_i p_i\to -\infty 
\quad \mbox{for each}\quad p'\in \R^d
\end{equation}
then any $F$-relaxed sub-solution satisfies a ``weak continuity''
condition along the junction hyperplane. Such a result is used when
reducing the set of test functions.  
\begin{lem}[``weak continuity'' condition on the junction hyperplane]\label{lem::r21}
  Assume that the Hamiltonians satisfy \eqref{assum:H} and that $F$
  satisfies \eqref{assum:F} and \eqref{eq::r20}. Then any relaxed
  sub-solution $u$ of \eqref{eq:hj-f} satisfies the following ``weak
  continuity'' property
\begin{equation}\label{eq::r21}
u^*(t,X)=\limsup_{(s,Y)\to (t,X),\ Y\in J_i^*}u(s,Y) \quad \mbox{for all}\quad i=1,\dots,N,\quad \mbox{for all}\quad (t,X)\in (0,T)\times \Gamma
\end{equation}
where we recall that  $J_i^*=J_i\backslash \Gamma$.
\end{lem}
The proof of this result is a straightforward adaptation of the one of
Lemma 2.3 in \cite{im} in the case $d=0$.

As in \cite{im}, we will see that the ``weak continuity'' property is
an important condition to avoid pathological relaxed sub-solutions
(that do exist) when $F$ is not semi-coercive.  Moreover it turns out
that the notion of ``weak continuity'' is stable, as shown in the
following result.
\begin{pro}[Stability of the weak continuity property]\label{pro::r23}
Consider a family of Hamiltonians $H^\varepsilon_i$ satisfying  \eqref{assum:H}. We also assume that the coercivity of the Hamiltonians is uniform in $\varepsilon$. Let $u^\varepsilon$ be a family of subsolutions of 
$$u_t + H^\varepsilon_i(Du)=0 \quad \mbox{in}\quad (0,T)\times J_i^*$$
for all $i=1,\dots,N$, and that $u^\varepsilon$ satisfies the ``weak continuity'' property \eqref{eq::r21}.
If $\bar u=\limsup{}^* u^\varepsilon$ is everywhere finite, then $\bar u$ still satisfies the ``weak continuity'' property \eqref{eq::r21}.
\end{pro}
The proof of this result is also a straightforward adaptation of the
one of Proposition 2.6 in \cite{im} in the case $d=0$.

\subsubsection{A reduced set of test functions}

We recall that the function $H_i^+$ is defined by
\[H_i^+(p',p_i)=\begin{cases}
H_i(p',\pi_i^0(p'))  &\quad \mbox{if}\quad p_i < \pi_i^0(p'),\\
H_i(p',p_i) &\quad \mbox{if}\quad p \ge \pi_i^0(p')
\end{cases}\] 
and the functions  $\pi_i^\pm : \R^d \times \R \to \R$ are defined  for $\lambda\ge A_i(p')=\min H_i(p',\cdot)$ as
\begin{align*}
\pi_i^+ (p',\lambda) &= \inf \{  p_i : H_i (p',p_i) = H_i^+ (p',p_i) =
\lambda \} \\
  \pi_i^- (p',\lambda) &= \sup \{  p_i : H_i (p',p_i) = H_i^- (p',p_i) =
\lambda \}. 
\end{align*}
\begin{defi}[Reduced  solutions -- the flux-limited case]\label{defi:equivalent}
  Assume the Hamiltonians satisfy \eqref{assum:H} and consider a
  continuous flux limiter  $A: \R^d \to \R$ such that for all
  $p' \in \R^d$, \( A (p') \ge A_0 (p') .\) Given $u\colon[0,T)\times
  J\to \R$ locally bounded, the function $u$ is a \emph{reduced
    sub-solution} (resp. \emph{reduced super-solution}) of \eqref{eq:hj-f}
  with $F=F_A$ in $J_T$ if and only if $u$ is a sub-solution
  (resp. super-solution) outside $\Gamma$ and for all test function
  $\varphi\in C^1(J_T)$ touching $u$ from above at $(t_0,X_0) \in
  (0,+\infty) \times \Gamma$, of the following form
\[ \varphi (t,x',x) = \phi (t,x') + \phi_0(x) \]
with 
\[\begin{cases}\phi \in C^1 ((0,+\infty) \times \R^d)\\
 D'\phi (t_0,x'_0) = p'_0\end{cases}
\quad 
\begin{cases}\phi_0 \in C^1 (\R) \\
\partial_i \phi_0(0) = \pi_i^+
 (p'_0,A(p'_0))
\end{cases}
\]
we have
\[\varphi_t + F_A(D\varphi) \le 0  \quad (\text{resp.}\quad \ge 0).\]
\end{defi}
\begin{pro}[Equivalence of Definitions~\ref{defi:fl} and
  \ref{defi:equivalent} under ``weak continuity'']\label{pro::r29}
  Every reduced super-solution (resp. subsolution) $u$ in the sense of
  Definition ~\ref{defi:fl} is also, for Definition
  \ref{defi:equivalent}, a flux-limited super-solution (resp. a
  flux-limited subsolution if $u$ satisfies moreover the
  "weak-continuity" property \eqref{eq::r21}).
\end{pro}
\begin{proof}
  It is clear that flux-limited sub-solutions (resp. super-solutions)
  are reduced sub-solutions (resp. reduced super-solutions). To prove
  that the converse holds true, we proceed as in \cite{im} by
  considering critical slopes in $x$. Precisely, it is enough to prove
  the following lemmas.
\begin{lem}[Critical slopes for super-solutions]\label{lem:critical}
Let $u$ be a super-solution of \eqref{eq:hj-FA} away from $\Gamma$ and let $\varphi$
touch $u_*$ from below at $P_0 =(t_0,X_0)$ with $X_0 \in \Gamma$. Then the ``critical
slopes'' defined as follows
\[ \bar p_i = \sup \{ \bar p \in \R_+: \exists r>0, \varphi (t,X) + \bar p x \le u_*
(t,X) \text{ for } (t,X) \in B_r (P_0) \cap \left((0,+\infty) \times J_i \right)\} \]
satisfy for all $i=1,\dots, N$, 
\[ \varphi_t (P_0) + H_i (D'\varphi (P_0), \partial_i \varphi (P_0) +
\bar p_i) \ge 0,\]
with the convention for $\bar p_i=+\infty$, that $H_i(p',+\infty)=+\infty$.
\end{lem}
\begin{lem}[Critical slopes for sub-solutions]\label{lem:criticalsub}
Let $u$ be a sub-solution of \eqref{eq:hj-FA} away from $\Gamma$ and let $\varphi$
touch $u^*$ from above at $P_0 =(t_0,X_0)$ with $X_0 \in \Gamma$. Then the ``critical
slopes'' defined as follows
\[ \underline p_i = \inf \{ \bar p \in \R_-: \exists r>0, \varphi (t,X) + \bar p x \ge u^*
(t,X) \text{ for } (t,X) \in B_r (P_0) \cap \left((0,+\infty) \times J_i\right) \} \]
satisfy for all $i=1,\dots, N$, 
\[ \varphi_t (P_0) + H_i (D'\varphi (P_0), \partial_i \varphi (P_0) +
\underline p_i) \le 0\quad \mbox{if}\quad \underline p_i>-\infty.\]
Moreover, we have
$$\underline p_i>-\infty\quad\mbox{for each}\quad i=1,\dots,N$$
if $u$ satisfies the ``weak continuity'' property \eqref{eq::r21}.
\end{lem}
\begin{rem}\label{rem::r25}
  Even if Lemma \ref{lem:criticalsub} is not stated this way, a close
  look at its proof shows that it is sufficient to have the ``weak
  continuity'' property pointwise at $(t_0,X_0)$ and on a single
  branch $J_i^*$ to prove that $\underline p_i>-\infty$ for the same index
  $i$.
\end{rem}

The proofs of these lemmas are straightforward adaptations of the
corresponding ones in \cite{im} so we skip them. The remainder of the
proof is also analogous and we also skip it. 
\end{proof}

\subsection{Effective junction conditions}
\label{s23}

\begin{defi}[Effective flux limiter $A_F$]\label{rem:def-AF}
  Let $p_i^0 \ge \pi_i^0(p')$ be minimal such that
  $H_i (p',p_i) = A_0$ and let $p^0$ denote $(p_1^0,\dots,p_N^0)$.
  The function $A_F$ is referred to as the \emph{effective flux
    limiter} and is defined as follows: for each $p' \in \R^d$, if
  $F(p',p^0) \le A_0(p')$, then $A_F(p')=A_0(p')$, else $A_F(p')$ is
  the only $\lambda \in \R$ such that
  $\lambda \ge A_0(p')= \max_i A_i (p')$ and there exists
  $p_i^+ \ge p_i^0$ such that
\[ H_i (p', p_i^+) = F (p', p^+ ) = \lambda \]
where $p^+ = (p_1^+,\dots,p_N^+).$ 
\end{defi}
\begin{rem}
  Notice that if $F$ satisfies \eqref{assum:F} then $\lambda$ is
  unique. But $p^+$ may be not unique.
\end{rem}
\begin{theo}[General junction conditions reduce to flux-limited ones]\label{th:class}
Let the Hamiltonians satisfy \eqref{assum:H} and let $F:\R^N\to \R$
satisfy \eqref{assum:F}. There exists a unique coercive continuous
function $A_F:\R^d \to \R$, 
satisfying $A_F\ge A_0$ with $A_0$ defined in \eqref{eq:A0},  such that the following holds.
\begin{enumerate}[i)]
\item Every $F$-relaxed super-solution (resp. sub-solution satisfying
  moreover the ``weak continuity'' property \eqref{eq::r21}) of
  \eqref{eq:hj-f} is a $A_F$-flux-limited super-solution
  (resp. sub-solution) of \eqref{eq:hj-FA}.
\item Conversely, every $A_F$-flux-limited super-solution (resp. sub-solution) of \eqref{eq:hj-FA},
is a $F$-relaxed  super-solution (resp. sub-solution) of \eqref{eq:hj-f}.
\item If $F$ is quasi-convex, so is $A_F$. 
\end{enumerate}
\end{theo}
\begin{proof}
  With the notation of Remark~\ref{rem:def-AF} in hand, we first
  recall that if $F(p',p^0)\ge A_0(p')$, then
  there exists only one $\lambda \ge A_0(p')$ such that there
  exists $p^+ = (p_1^+,\dots,p_N^+)$ with $p_i^+ \ge p_i^0$ such that 
\[ H_i(p',p^+_i) = F(p',p^+) = \lambda. \]

The coercivity of $A_F$ is a direct consequence of the fact that $A_F
\ge A_0$. We thus prove next that $A_F$ is continuous. Consider a
sequence $(p'_n)_n$ converging towards $p'$. Then we have two cases.

\textsc{Case 1.}
There exists $p^+_n
=(p^+_{1,n}, \dots, p^+_{N,n})$ with $p_{i,n}^+ \ge p_i^0=p_i^0(p_n')$ such
that
\begin{equation}\label{eqn} 
H_i (p'_n,p^+_{i,n}) = F(p'_n,p^+_n) = A_n =A_F (p'_n)\ge A_0(p_n') \quad \mbox{if}\quad F(p_n',p^0(p_n'))\ge A_0(p_n').
\end{equation}
 We can pass to the limit in  \eqref{eqn} and get
\[ H_i (p',p^+_i) = F(p',p^+) = A \ge A_0(p')\]
with $p_i^+\ge p_i^0(p')$
and then $A = A_F (p')$.

\textsc{Case 2.}
$$A_n=A_0(p_n')=A_F (p'_n) \quad \mbox{if}\quad F(p_n',p^0(p_n'))\le A_0(p_n').$$
 We first claim that $(p^+_{i,n})_n$ is
bounded. Indeed, if not, then $A_n \to +\infty$ and, for $n$ large
enough, 
\[ F(p'_n,p^0(p_n')) \ge A_n \]
which is impossible. The claim also implies that $(A_n)_n$ is also
bounded. Consider now two converging subsequences, still denoted by
$(p_n')_n$ and $(A_n)_n$, and let $p'$ and $A$ be their limits.
We get
\[A=A_0(p')\]

If $F(p',p^0(p'))\le A_0(p')$,  then $A_F(p')=A_0(p')=A$.

If $F(p',p^0(p'))> A_0(p')$, then we have to enter in more details in
the results of the limit process.  We get
\[
F(p',\bar p^0)\le A_0(p') \quad \mbox{and}\quad A=A_0(p')=H_i(p',\underline p_i^0) 
\quad \mbox{where}\quad \underline p_i^0\ge \pi_i^0(p')
\]
with
$$\bar p^0 = \lim \ p^0(p_n') \quad \mbox{for a subsequence}$$
which implies $\underline p_i^0 \ge p_i^0(p')$. Then we can choose
some $p_i^+\in [p_i^0(p'),\underline p_i^0]$ such that
$$H_i(p',p_i^+)=F(p',p^+)=A_0(p')=A$$
which shows again that $A_F(p')=A$.
This ends the proof that $A_F$ is contiuous.
\medskip

\noindent {\it Proof of i).} We only do the proof for sub-solutions
since the proof for super-solutions follows along the same lines.  Let
$\varphi$ be a test function touching $u^*$ from above at
$P_0=(t_0,X_0)$. We only need to consider the case where
$X_0 \in \Gamma$. From Proposition \ref{pro::r29}, we can also assume
that
\[ \varphi (t,X) = \phi (t,x') + \phi_0(x) \]
with 
\[ D'\phi (t_0,x'_0) = p'_0 \quad \text{ and } \quad \partial_i \phi_0
(0) = \pi_i^+ (p'_0,A_F (p'_0)).\]
We have 
\[ \varphi_t (P_0) + \min (F(D \varphi (P_0)), \min_i H_i
(D' \varphi(P_0),\partial_i\varphi(P_0)) \le 0 \]
which yields
\[ \varphi_t (P_0) + \max (F(p'_0,\pi^+(p'_0,A_F(p'_0))), A_F (p'_0))
\le 0. \]
In view of the definition of $A_F$, we get
\[ \varphi_t (P_0) + A_F (p'_0) \le 0.\]
Now compute 
\[ F_{A_F} (D \varphi (P_0)) = \max (A_F(p'_0), \max_i H_i^-
(p'_0,\pi_i^+ (p'_0,A_F (p'_0))) = A_F(p'_0). \]
This ends the proof of i).

\noindent {\it Proof of ii).} We only do the proof for super-solutions
since the proof for sub-solutions follows along the same lines. Let
$\varphi$ be a test function touching $u_*$ from below at
$P_0=(t_0,X_0)$.  We want to show that it is a $F$-relaxed 
supersolution, i.e.
\begin{equation}\label{eq::r32}
\max(F(D \varphi (P_0)), \max_i H_i(D' \varphi(P_0),\partial_i\varphi(P_0)) \ge \lambda := -\varphi_t(P_0).
\end{equation}
We set
$$D \varphi (P_0)=(p_0',p) \quad \mbox{with}\quad p=(p_1,\dots,p_N).$$
We know that $u$ is a $F_{A}$-reduced  solution with $A=A_F$, i.e.
\begin{equation}\label{eq::r33}
\max(A_F(p_0'), \max_i H_i^-(p_0',p_i))=F_{A_F}(D \varphi (P_0))\ge \lambda.
\end{equation}
Moreover, we have
\begin{equation}\label{eq::r30}
F(p_0',\pi^+(p_0',A_F(p_0')))=A_F(p_0')> A_0(p_0')
\end{equation}
or
\begin{equation}\label{eq::r31}
A_F(p_0')= A_0(p_0').
\end{equation}
We now distinguish two cases.

\textsc{Case 1.} 
Assume first that there exists an index $i_0$ such that 
\[ 
H_{i_0}(p_0',p_{i_0})\ge \max (A_F(p_0'), \displaystyle \max_i H_{i}(p_0',p_{i})).
\]
Then \eqref{eq::r33} implies the result \eqref{eq::r32}.

\textsc{Case 2.}
Assume that for all $i$, we have $H_i(p_0',p_i)< A_F(p_0')$. Then $p_i< \pi_i^+(p_0',A_F(p_0'))$ and
$F(p_0',p_i)\ge F(p_0',\pi^+(p_0',A_F(p_0'))=A_F(p_0')\ge \lambda$ in case of \eqref{eq::r30}.

In the case of  \eqref{eq::r31}, we have $A_F(p_0')=A_0(p_0')$ and the inequality for all $i$
$$H_i(p_0',p_i)< A_F(p_0')=A_0(p_0')=\max_j \left(\min_{q_j} H_j(p_0',q_j) \right)$$
leads to a contradiction.
The proof of ii) is now complete.
\medskip

\noindent {\it Proof of iii).} It follows from Proposition \ref{pro:A-convex} below.
The proof is now complete.
\end{proof}

We now turn to the following useful proposition.
\begin{pro}[Quasi-convex effective flux limiters]\label{pro:A-convex}
  If the Hamiltonians $H_i$ satisfy \eqref{assum:H} and the flux
  function $F$ satisfies \eqref{assum:F}-\eqref{assum:F-convexity},
  then $A_F$ is continuous, quasi-convex and coercive.
\end{pro}
Before proving Proposition~\ref{pro:A-convex}, we state and prove the
following elementary lemma. 
\begin{lem}[Quasi-convexity of the functions $A_i$]\label{lem:ai-convex}
  If the Hamiltonians $H_i$ are quasi-convex (resp. convex), continuous and coercive,
  so are the functions $A_i$ defined in (\ref{eq:A0}). In particular, $A_0 = \max_i A_i$ is
  quasi-convex (resp. convex), continuous and coercive.
\end{lem}
\begin{proof}
  We only address the question of the quasi-convexity of the functions
  $A_i$ since their continuity and  coercivity are simpler.

Consider $p'$ and $q'$ such that $A_i(p') \le \lambda$ and $A_i(q')
\le \lambda$ for some $\lambda \in \R$. There exists $p_i,q_i \in \R$
such that 
\[ A_i(p') = H_i (p',p_i) \quad A_i(q') = H_i (q',q_i).\]
Then $(p',p_i),(q',q_i) \in \{ H_i \le \lambda \}$ and we conclude
from the convexity of $\left\{H_i\le \lambda\right\}$ that for $t,s \ge 0$ with $t+s=1$, 
\[ A_i (tp'+sq') \le H_i (tp'+sq',tp_i + s q_i)  
\le \lambda.\]
This achieves the proof of the lemma.
\end{proof}

\begin{proof}[Proof of Proposition~\ref{pro:A-convex}]
  We assume that the Hamiltonians $H_i$ are convex,
  $p_i \mapsto H_i(p',p_i)$ is increasing in $[\pi^0_i(p'),+\infty)$
  and decreasing in $(-\infty,\pi^0_i(p')]$ and $F$ is convex in all
  variables and $p\mapsto F(p',p)$ is decreasing in each variable for
  every $p'$ fixed. In particular, the functions $\pm \pi_i^\pm$ are
  concave.  The general case follows by an approximation argument and
  by remarking that it is enough to find $\beta$ increasing such that
  $\beta \circ F$ and $\beta \circ H_i$ satisfy the previous
  assumptions (see Lemma \ref{lem::1}).

We now prove that 
\[ G(p',\lambda) = F (p',\pi^+(p',\lambda)) \]
is convex w.r.t. $(p',\lambda) \in \epi A_0$. 
For $(p',\lambda), (q',\mu) \in \epi A_0$ and $t,s \ge 0$ with
$t+s=1$, we can use the monotonicity of $F$ together with the
concavity of $\pi_i^+$ (see Lemma~\ref{lem:pi-pm}) to get
\begin{align*}
  t G(p',\lambda) + s G(q',\mu) & \ge F(tp'+sq', t\pi^+(p',\lambda)+s
  \pi^+(q',\mu)) \\
& \ge F(tp'+sq', \pi^+(tp'+sq',t\lambda+s \mu)) \\
& = G(tp'+sq',t\lambda +s \mu). 
\end{align*}
Similarly, we can see that $G$ is non-increasing with respect to
$\lambda$. 

We next remark that
\[ A_F (p') = G (p',A_F(p')) \]
and for $p',q' \in \R^d$ and $t,s \ge 0$ with $t+s=1$, we can write 
\begin{align*}
t A_F (p') + s A_F (q') & = t G(p',A_F(p')) + s G(q',A_F(q')) \\
& \ge G (tp'+sq', t A_F(p')+ sA_F (q'))
\end{align*}
and 
\[ A_F(tp'+sq') = G(tp'+sq',A_F(tp'+sq')). \]
We thus deduce from the monotonicity of $G$ in $\lambda$ that 
\[ A_F (tp'+sq') \le t A_F(p')+ s A_F (q').\]
The proof is now complete.
\end{proof}

\subsection{Minimal/maximal Ishii solutions}
\label{sec:ishii}

In this section, we extend the study of Ishii solutions started in
\cite{im} to a multi-dimensional setting. The proofs are straightforward
extensions of the one contains in \cite{im} but we provide them for the
sake of completeness. 

We are interested in the following Hamilton-Jacobi equations posed in $\R^{d+1}$
\begin{equation}\label{eq:hj-md}
\begin{cases}
U_t +  H_L (DU) = 0, & t>0, X=(x',x_{d+1}), x_{d+1}<0, \\
U_t + H_R (DU) = 0, & t>0, X=(x',x_{d+1}), x_{d+1} >0.
\end{cases}
\end{equation} 
We recall that Ishii solutions are viscosity solutions of
\eqref{eq:hj-md} in $\R^{d+1} \setminus \{ x_{d+1}=0\}$ such that,
\begin{equation}\label{eq:cond-ishii}
\begin{cases}
U_t + \max(H_L(DU),H_R(DU)) \ge 0, & t>0, x_{d+1}=0 \\
U_t + \min (H_L(DU),H_R(DU)) \le 0, & t>0, x_{d+1}=0
\end{cases}
\end{equation}
(in the viscosity sense).  The Hamilton-Jacobi
equation~\eqref{eq:hj-md} posed in $\R^{d+1}$ is naturally associated
with another HJ equation posed on a multi-dimensional junction with
$N=2$ ``branches'' (or ``sheets''). Indeed, if we define for $(x',x_i) \in J_i$, 
\begin{equation}\label{eq:U2u}
 u (t,(x',x_i)) = \begin{cases}
U(t,(x',-x_i)) & \text{ if } i=1, \\
U (t,(x',x_i)) & \text{ if } i=2,
\end{cases}
\end{equation}
then $u$ is a solution of \eqref{eq:hj-f} in $J \setminus \Gamma$ with  
\begin{equation}\label{eq:Ham}
 H_1 (p',p_1) = H_L(p',-p_1) \quad \text{ and } \quad H_2 (p',p_2) = H_R (p',p_2).
\end{equation}
Conversely, if $u$ is a solution of \eqref{eq:hj-f} posed in $J$ with $N=2$, and 
$u^i$ denotes $u_{|(0,T) \times J^i}$, then the function $U$ defined by
\begin{equation}\label{eq:u2U}
 U(t,(x',x_{d+1})) = \begin{cases} u^1 (t,(x',-x_{d+1})) & \text{ for } x_{d+1} < 0 \\
u^2 (t,(x',x_{d+1})) & \text{ for } x_{d+1} >0 
\end{cases}
\end{equation}
satisfies \eqref{eq:hj-md} in $\R^{d+1}$.

\begin{pro}[Minimal/maximal Ishii solutions in the Euclidian setting]\label{prop:fl-is}
  The maximal (resp. minimal) Ishii solution $U^\pm$ of \eqref{eq:hj-md} corresponds to the
  $A_I^\mp$-flux-limited solution $u^\pm$ of \eqref{eq:hj-f} with Hamiltonians given by 
\eqref{eq:Ham} and 
\begin{align*}
 A_I^+ (p')&= \max (A_0(p'), A^* (p')) \\
A_I^- (p') & = 
\begin{cases}
 A_I^+ (p') & \text{ if } \pi_R^0 (p') < \pi_L^0(p') \\
 A_0 (p') & \text{ if } \pi_R^0 (p') \ge \pi_L^0(p').
\end{cases}
\end{align*}
where
\[ A^* (p')= \max_{p_{d+1} \in [\pi_L^0 (p')\wedge \pi_R^0(p'),\pi_L^0 (p')\vee \pi_R^0(p')]} H_R (p',p_{d+1}) \wedge H_L (p',p_{d+1}).\]
\end{pro}
\begin{rem}
  The paper \cite{im} contains a much more complete study of Ishii
  solutions in the one-dimensional setting.  Even if such a study most
  probably extends to the multi-dimensional setting, we focus here in
  the identification of the minimal and the maximal Ishii solutions. 
  Such a result is used in \cite{in}.
\end{rem}
The proof of Proposition~\ref{prop:fl-is} is very similar to the one 
in \cite{im} for the one-dimensional setting. We give details for
the reader's convenience. 

The proof relies on the following lemma, which is the analogue of
\cite[Lemma~2.18]{im}. Since the proof follows along the same lines,
we skip it.
\begin{lem}[``weak continuity'' condition with $C^1$ test
  functions] \label{lem:wc-c1} Given two Hamiltonians $H_L$, $H_R$
  satisfying \eqref{assum:H} and $H_0$ continuous and coercive
  (i.e. $\lim_{|P|\to +\infty} H_0(P)= +\infty$), let
  $u : (0,T) \times \R^{d+1} \to \R$ be upper semi-continuous such that
  every $C^1$ function $\phi$ touching $u$ from above at $(t,X)$ with
$X=(x',x_{d+1})$ and $t>0$,
  satisfies
\[\begin{cases} 
\phi_t + H_L (D\phi ) \le 0 & \text{ if }  x_{d+1} <0, \\
\phi_t  + H_R (D\phi ) \le 0 & \text{ if }  x_{d+1} >0, \\
\phi_t + H_0 (D\phi) \le 0 & \text{ if }  x_{d+1} =0.
\end{cases}\]
Then for all $t \in (0,T)$ and $X=(x',0)$, 
\[ u (t,X) = \limsup_{(s,Y) \to (t,X), y_{d+1}>0} u(s,Y) 
= \limsup_{(s,Y) \to (t,X), y_{d+1}<0} u(s,Y)\]
where $Y=(y',y_{d+1})$. 
\end{lem}
\begin{proof}[Proof of Proposition~\ref{prop:fl-is}]
We have to prove the four following assertions:
\begin{enumerate}[i)]
\item \label{i} every $F_{A_I^\pm}$-flux-limited sub-solution  corresponds to a Ishii sub-solution; 
\item \label{ii} every $F_{A_I^\pm}$-flux-limited super-solution  corresponds to a Ishii super-solution; 
\item \label{iii} every Ishii sub-solution corresponds to an $F_{A_I^-}$-flux-limited sub-solution;
\item \label{iv} every Ishii super-solution corresponds to an  $F_{A_I^+}$-flux-limited super-solution.
\end{enumerate}
In order to prove these assertions, it is convenient to translate the
notion of $A$-flux-limited solution to the Euclidian setting. It
reduces to replace $F_A$ with $\check{F}_A$ where
\[ \check{F}_A (p',p_L,p_R) = \max (A (p'), H_L^+(p',p_L),H_R^-(p',p_R)) \]
where $H_L^+ (p',p_L) = H_1^-(p',-p_L)$ is the non-decreasing part of $p_L \mapsto H_L(p',p_L)$. 
In particular, $H_L^+(p',p_L) = H_L(p',p_L)$ if $p \ge \pi_L^0=p_L^+$. In the same way, $H_R^-(p',p_R) = H_R (p',p_R)$
if $p \le \pi_R^0=p_R^-$. 
\medskip

Let $\phi \in C^1 ((0,+\infty) \times \R^{d+1})$ be a test function
touching a $\check{F}_{A_I^\pm}$-flux-limited sub-solution from above at $X$. We have 
\[ \check{F}_A (p',p,p) \le \lambda \]
where $p' = D' \phi (X)$, $p = \partial_{d+1} \phi (X)$ and $\lambda = -\phi_t (X)$. This means
\[ \max (A_I^\pm (p'), H_L^+ (p',p), H_R^-(p',p)) \le \lambda.\]
If $p \le p_R^-$ or $p \ge p_L^+$, then $H_R^-(p',p)= H_R(p',p)$ or $H_L^+ (p',p) = H_L(p',p)$ and we get
\[ \min (H_R (p',p),  H_L(p',p)) \le \lambda. \]
If now $p_L^+ < p < p_R^-$, then $A_I^+(p')=A_I^-(p') \ge A^*(p')$ and 
\[ \min (H_R(p',p),H_L(p',p)) \le A^* (p') \le A_I^\pm \le \lambda \]
and we conclude in this case too. This achieves the proof of \ref{i}). 
\medskip

In order to prove \ref{ii}), we remark that 
\[ A_I^+ (p') \le \max (H_L (p',p), H_R(p',p)). \]
Let $\phi \in C^1 ((0,+\infty) \times \R^{d+1})$ be a test function
touching a $\check{F}_{A_I^\pm}$-flux-limited super-solution from below at $X$.
We have in this case
\[ \max (A_I^\pm (p'), H_L^+ (p',p), H_R^-(p',p)) \ge \lambda.\]
Since $A_I^- \le A_I^+$, we get immediately that 
\[ \max (H_L (p',p), H_R(p',p)) \ge \lambda.\]
This achieves the proof of \ref{ii}). 
\medskip

We next prove \ref{iii}). First, the weak continuity condition at
$x_{d+1}=0$ holds true thanks to Lemma~\ref{lem:wc-c1}.  Then we can
apply Proposition~\ref{pro::r29} and consider a test function
$\phi \in C((0,+\infty) \times \R^{d+1})$ such that
$\phi_{|\{\pm x_{d+1} \ge 0\}}$ are $C^1$ and
\begin{align*}
p_L= \partial_{d+1} \phi (x',0-) = \pi_L^- (p',A_I^- (p'))\\
p_R = \partial_{d+1} \phi (x',0+) = \pi_R^+ (p',A_I^- (p')).
\end{align*}
Assume that $\phi$ touches an Ishii sub-solution at a point $X$. Let
$p'= D'\phi (X)$. If $A_I^-(p')=A_0(p')$, then we can argue as in
\cite[Theorem~2.7,i)]{im} and get the desired result. We thus assume
that $A_I^- (p')= A_I^+ (p')= A^*(p') = H_L(p',p^*)=H_R(p',p^*)$ with
$p^* \in [\pi_R^0(p'),\pi_L^0 (p')]$. But in this case 
\[ p_L = p_R =p^* \]
and the test function 
\[ \phi (t,x',x) = \varphi (t,x') + p^* x_{d+1}\]
is $C^1$ in $(0,+\infty) \times \R^{d+1}$. 
In particular, since $u$ is an Ishii sub-solution, we get
\[ A_I^-(p')= \min (H_L (p',p^*),H_R(p',p^*)) \le \lambda\]
which yields the desired inequality (this can be checked easily).
This achieves the proof of \ref{iii}).  \medskip

We finally prove \ref{iv}). We use once again the reduced set of test
functions and consider $\phi$ of the form
\[ \phi (t,x',x_{d+1}) = \varphi (t,x') + \phi_0 (x_{d+1})\]
with 
\[ \phi_0'(0+) = \pi^+_R(p',A_I^+ (p')) \quad \text{ and } \quad
\phi_0'(0-) = \pi_L^-(p',A_I^+(p'))\]
where $p' = D' \varphi (t_0,x'_0)$ if $\phi$ touches the Ishii
super-solution $u$ from below at $(t_0,(x'_0,0))$. 

If $A_I^+ (p') = A^* (p') \ge A_0 (p')$, then we choose
\(\phi_0(x_{d+1}) = p^* x_{d+1} \)
with $p^*$ such that $A^*(p')=H_R(p',p^*)=H_L(p',p^*)$.
Since $u$ is an Ishii super-solution, we have
\[ \varphi_t (t_0,x_0') + \max (H_R (p', p^*), H_L (p',p^*)) \ge 0 \]
that is to say 
\[ \varphi_t (t_0,x_0') + A_I^+ (p') \ge 0 \]
which is the desired inequality. 

If now $A_I^+ (p') = A_0 (p') \ge A^* (p')$, then we choose 
\[ \phi_0(x_{d+1}) = \pi_R^+ (p',A_0(p')) x_{d+1} \un_{x_{d+1}\ge 0} + \pi_L^- (p',A_0(p')) x_{d+1} \un_{x_{d+1}\le 0}. \]
We notice that there exists $\alpha \in \{R,L\}$ such that 
$A_0 (p') = H_\alpha (p', \pi_\alpha^0(p'))$
and
\[ \pi_L^-(p',A_0(p')) \le \pi_R^+(p',A_0(p'))\]
and one of them equals $\pi_\alpha^0(p')$. 
These three facts imply that
\[ \tilde \phi (t,x',x_{d+1}) := \varphi (t,x') + \pi_\alpha^0(p')
x_{d+1} \le \phi (t,x,',x_{d+1}).\]
In particular $\tilde \phi$
is a $C^1$ test function touching $u$ from below at $(t_0,(x'_0,0))$. Since 
$u$ is an Ishii super-solution, we get in this case,
\[ \varphi_t (t_0,x_0') + \max (H_R (p',\pi_\alpha^0(p')), H_L (p',\pi_\alpha^0(p'))) \ge 0 \]
which implies 
\[ \varphi_t (t_0,x_0') + A_0 (p') \ge 0.\]
The proof is now complete.
\end{proof}

\paragraph{Aknowledgements. }
This work was partially supported by the ANR-12-BS01-0008-01
HJnet project.

\providecommand{\href}[2]{#2}
\providecommand{\arxiv}[1]{\href{http://arxiv.org/abs/#1}{arXiv:#1}}
\providecommand{\url}[1]{\texttt{#1}}
\providecommand{\urlprefix}{URL }

\end{document}